\documentclass[12pt]{amsart}
\usepackage{cases}
\usepackage{mathrsfs}
\usepackage{color}
\usepackage{latexsym,amssymb}
\usepackage{a4wide}
\usepackage{amsthm}
\usepackage{amsmath}
\usepackage{graphicx}
\input xy
\xyoption{all}
\usepackage{verbatim}
\usepackage[lite,abbrev,msc-links]{amsrefs}
\usepackage{framed}
\usepackage{bm}

\numberwithin{equation}{section}
\begin{document}
	\theoremstyle{plain}
	\newtheorem{thm}{Theorem}[section]
	\newtheorem{lem}[thm]{Lemma}
	\newtheorem{cor}[thm]{Corollary}
	\newtheorem{cor*}[thm]{Corollary*}
	\newtheorem{prop}[thm]{Proposition}
	\newtheorem{prop*}[thm]{Proposition*}
	\newtheorem{conj}[thm]{Conjecture}
	\theoremstyle{definition}
	\newtheorem{construction}{Construction}
	\newtheorem{notations}[thm]{Notations}
	\newtheorem{question}[thm]{Question}
	\newtheorem{prob}[thm]{Problem}
	\newtheorem{rmk}[thm]{Remark}
	\newtheorem{remarks}[thm]{Remarks}
	\newtheorem{defn}[thm]{Definition}
	\newtheorem{claim}[thm]{Claim}
	\newtheorem{assumption}[thm]{Assumption}
	\newtheorem{assumptions}[thm]{Assumptions}
	\newtheorem{properties}[thm]{Properties}
	\newtheorem{exmp}[thm]{Example}
	\newtheorem{comments}[thm]{Comments}
	\newtheorem{blank}[thm]{}
	\newtheorem{observation}[thm]{Observation}
	\newtheorem{defn-thm}[thm]{Definition-Theorem}
	\newtheorem*{Setting}{Setting}

	\newcommand{\sA}{\mathscr{A}}	\newcommand{\sB}{\mathscr{B}}	\newcommand{\sC}{\mathscr{C}}	\newcommand{\sD}{\mathscr{D}}	\newcommand{\sE}{\mathscr{E}}	\newcommand{\sF}{\mathscr{F}}	\newcommand{\sG}{\mathscr{G}}	\newcommand{\sH}{\mathscr{H}}	\newcommand{\sI}{\mathscr{I}}	\newcommand{\sJ}{\mathscr{J}}	\newcommand{\sK}{\mathscr{K}}	\newcommand{\sL}{\mathscr{L}}	\newcommand{\sM}{\mathscr{M}}	\newcommand{\sN}{\mathscr{N}}	\newcommand{\sO}{\mathscr{O}}	\newcommand{\sP}{\mathscr{P}}
	\newcommand{\sQ}{\mathscr{Q}}	\newcommand{\sR}{\mathscr{R}}	\newcommand{\sS}{\mathscr{S}}	\newcommand{\sT}{\mathscr{T}}	\newcommand{\sU}{\mathscr{U}}	\newcommand{\sV}{\mathscr{V}}	\newcommand{\sW}{\mathscr{W}}	\newcommand{\sX}{\mathscr{X}}	\newcommand{\sY}{\mathscr{Y}}	\newcommand{\sZ}{\mathscr{Z}}	\newcommand{\bZ}{\mathbb{Z}}	\newcommand{\bN}{\mathbb{N}}	\newcommand{\bQ}{\mathbb{Q}}	\newcommand{\bC}{\mathbb{C}}
	\newcommand{\bR}{\mathbb{R}}	\newcommand{\bH}{\mathbb{H}}	\newcommand{\bD}{\mathbb{D}}	\newcommand{\bE}{\mathbb{E}}	\newcommand{\bP}{\mathbb{P}}	\newcommand{\bV}{\mathbb{V}}	\newcommand{\cV}{\mathcal{V}}	\newcommand{\cF}{\mathcal{F}}	\newcommand{\bfM}{\mathbf{M}}	\newcommand{\bfN}{\mathbf{N}}	\newcommand{\bfX}{\mathbf{X}}	\newcommand{\bfY}{\mathbf{Y}}	\newcommand{\spec}{\textrm{Spec}}	\newcommand{\dbar}{\bar{\partial}}	\newcommand{\ddbar}{\partial\bar{\partial}}	\newcommand{\redref}{{\color{red}ref}}
	
	\title[Arakelov Type Inequalities] {Arakelov Type Inequalities and Deformation Boundedness of polarized varieties}
	
	\author[Junchao Shentu]{Junchao Shentu}
	\email{stjc@ustc.edu.cn}
	\address{School of Mathematical Sciences,
		University of Science and Technology of China, Hefei, 230026, China}

	\begin{abstract}
		We give two kinds of generalizations of Arakelov type inequalities for higher dimensional families. These results give higher dimensional generalizations (in both fibers and bases) of the weakly boundedness in Par\v{s}in-Arakelov's reformulation of the geometric Shafarevich conjecture. As a consequence, we deduce the following results.
		\begin{itemize}
			\item {\bf{Hyperbolicity}:} We give an alternative proof (using the theory of degeneration of Hodge structure) to the hyperbolicity (in Par\v{s}in-Arakelov's reformulation, i.e. Viehweg's hyperbolicity conjecture) of the base of a family with maximal variation whose general fibers admit good minimal models. This has been proved by Popa-Schnell Hodge theoretically.
			\item {\bf{Boundedness}:} We show the deformation boundedness of admissible families of lc stable minimal models (introduced by Birkar) with an arbitrary Kodaira dimension.
		\end{itemize}
	\end{abstract}
	
	\maketitle
	
	\section{Introduction}
Throughout this paper  every variety is defined over the complex number field $\bC$.
	Let $f:Y\rightarrow X$ be an algebraic fiber space from a smooth projective variety to a smooth projective curve. Par\v{s}in-Arakelov's reformulation \cite{Parsin1968,Arakelov1971}, of the geometric Shafarevich conjecture is the following package of statements (say, the \textit{Par\v{s}in-Arakelov-Shafarevich package}). 
	\begin{description}
		\item[Hyperbolicity](=Viehweg's hyperbolicity conjecture) Let $D\subset X$ be a divisor such that $f$ is smooth over $X\backslash D$. Assume that $f$ is non-isotrivial. Then $\omega_X(D)$ is big.
		\item[Weakly boundedness](=Arakelov type inequality) $\deg f_\ast(\omega_{Y/X}^{\otimes m})$ is bounded in terms of $g(B)$, $\# D$, $m$ and the Hilbert function for the general fibers.
		\item[Boundedness] There exist only finitely many deformation types of smooth families over $X\backslash D$ when the fibers form a bounded moduli.
		\item[Rigidity] There exists no non-trivial deformation of non-isotrivial smooth families over $X\backslash D$ (under some conditions).
	\end{description}
	When $f$ is a family of curves with genus larger than one, the Par\v{s}in-Arakelov-Shafarevich package is proved by Par\v{s}in \cite{Parsin1968} in the case when $D=\emptyset$ and by Arakelov \cite{Arakelov1971} in general. The arithmetic analogue of  Shafarevich's conjecture for curve fiberations over a number field has been confirmed by Faltings \cite{Faltings19832}. It, combined with Par\v{s}in's trick, implies Mordell's conjecture. 
	
	For families of canonically polarized manifolds of higher dimensions, the Par\v{s}in-Arakelov-Shafarevich package is proved by Bedulev-Viehweg \cite{Viehweg2000} (hyperbolicity and weakly boundedness) and Kov\'acs-Lieblich \cite{Kovacs2011} (boundedness). Liu-Todorov-Yau-Zuo \cite{LTYZ2005} and Viehweg-Zuo \cite{VZ2003_2} obtain the deformation boundedness of polarized Calabi-Yau manifolds. The rigidity part is generally false for higher dimensional families. Faltings \cite{Faltings1983} constructs non-rigid families of principally polarized abelian varieties and Liu-Todorov-Yau-Zuo \cite{LTYZ2005} construct non-rigid families of polarized Calabi-Yau manifolds. Besides, there are some sufficient conditions on families (\cite{VZ2003_2,Kovacs2005}) for the rigidity. Readers may see \cite{Kovacs2009} for a survey on this subject. 
	
    In addition to the rigidity, the other parts of the Par\v{s}in-Arakelov-Shafarevich package are expected to hold for algebraic fiber spaces over higher dimensional bases. The hyperbolicity over higher dimensional base is proved by Popa-Schnell \cite{PS17} for maximal variational smooth families whose general fibers admit good minimal models. The hyperbolicity problem for log smooth family of general type is studied by Wei-Wu \cite{WeiWu2023}. On the other hand, Kov\'acs-Lieblich \cite{Kovacs2011} shows that some Arakelov type inequalities imply the  deformation boundedness of families over higher dimensional bases when the relevant coarse moduli space admits a nice compactification. 
	
	The purpose of the present paper is to give two generalizations of the weakly boundedness (i.e. Arakelov type inequalities) over higher dimensional bases. One (Theorem \ref{thm_main_Arakelov_inequality}) is a generalization of the classical Arakelov inequality regarding the degree of the pushforward of pluri-canonical sheaf. The other (Theorem \ref{thm_main_numbound_moduli}) is a uniform numerical bound of the Koll\'ar type polarizations of the moduli space of lc stable minimal models which is recently introduced by Birkar \cite{Birkar2022}. As an application of Theorem \ref{thm_main_numbound_moduli}, we generalize the boundedness part of the Par\v{s}in-Arakelov-Shafarevich package to admissible families of lc stable minimal models of an arbitrary Kodaira dimension (see \S \ref{section_defbounded}). Both Arakelov type inequalities follow from a meta Arakelov type inequality (Theorem \ref{thm_abs_Arakelov}).
	\subsection{The first Arakelov type inequality}
	The first purpose of the present paper is to generalize Arakelov's inequality \cite{Arakelov1971} to families over higher dimensional bases. The main theorem is
	\begin{thm}\label{thm_main_Arakelov_inequality}
		Let $f:Y\to X$ be a proper surjective morphism from a complex manifold $Y$ to a smooth projective variety $X$ with relative dimension $n$. Let $D_f\subset X$ be an effective divisor such that $f$ is a K\"ahler submersion over $X\backslash D_f$. Then the following hold. 
		\begin{enumerate}
			\item Assume that there is a strictly semistable reduction in codimension one (\S \ref{section_ssred}) $\widetilde{f}:\widetilde{Y}\to\widetilde{X}$ of $f$ such that $\det \widetilde{f}_\ast(\omega_{\widetilde{Y}/\widetilde{X}}^{\otimes k})$\footnote{$\det(\sF)$ denotes the reflexive hull of $\wedge^{{\rm rk}(\sF)}\sF$ for a torsion free sheaf $\sF$.} is a big line bundle for some $k\geq1$. Then $\omega_{X}(D_f)$ is a big line bundle. 
			\item Let $W\subset f_\ast(\omega_{Y/X}^{\otimes k})^{\otimes r}$ be a coherent subsheaf for some $k,r\geq 1$. Assume that $\omega_{X}(D_f)$ is pseudo-effective. Then the following Arakelov type inequalities hold.
			\begin{align}\label{align_main_Arakelov_ineq_H}
				\frac{c_1(W)A_1A_2\cdots A_{d-1}}{{\rm rank}(W)}\leq rk\left(\frac{n}{2}c_1(\omega_X(D_f))+c_1(\sO_X(R_f))\right)A_1A_2\cdots A_{d-1}
			\end{align}
			for any semiample effective divisors $A_1,\dots, A_{d-1}$ $(d:=\dim X)$ on $X$, and
			\begin{align}\label{align_main_Arakelov_ineq_alpha}
				\mu_\alpha(W)\leq rk\left(n\mu_\alpha(\omega_X(D_f))+\mu_\alpha(\sO_X(R_f))\right)+\frac{\mu_\alpha(\sO_X(D_f))}{{\rm rank}W}
			\end{align}
			for every movable class $\alpha\in N_1(X)$\footnote{$\mu_\alpha(F):=\frac{c_1(F)\cdot\alpha}{{\rm rank}F}$ for every torsion free coherent sheaf $F$.}. Here $R_f\subset D_f$ is the ramified divisor of $f$ (c.f. \S \ref{section_Arakelov_inequality}).
		\end{enumerate}
	\end{thm}
    When $Y$ is a smooth projective variety, the assumption in Claim (1) is valid when the general fibers of $f$ are of general type (Koll\'ar \cite{Kollar1987}) or admit good minimal models (Kawamata \cite{Kawamata1985}). One of the consequences of Theorem \ref{thm_main_Arakelov_inequality}-(1) is an alternative proof of Popa-Schnell \cite{PS17}'s result on Viehweg's hyperbolic conjecture for families of projective manifolds which admit good minimal models.
    The main difficulty of loc.~cit.~is to construct the Viehweg-Zuo sheaf using Hodge modules. In the present paper we give an alternative construction of the Viehweg-Zuo sheaf using the analytic prolongations (in the sense of Simpson \cite{Simpson1990} and Mochizuki \cite{Mochizuki20072}) of Viehweg-Zuo's original constructions. 

	When $\dim X=1$, the inequalities (\ref{align_main_Arakelov_ineq_H}) and (\ref{align_main_Arakelov_ineq_alpha}) are effective versions of Viehweg-Zuo's \cite{VZ2001} Arakelov type inequality. 
	When $r=1$, $\dim X=1$ and $f$ is a strictly semistable family, the inequality (\ref{align_main_Arakelov_ineq_H}) is obtained by Viehweg-Zuo \cite{VZ2006} and M\"{o}ller-Viehweg-Zuo \cite{MVZ2006}. Our proof of (\ref{align_main_Arakelov_ineq_H}) is deeply influenced by their works. (\ref{align_main_Arakelov_ineq_H}) is optimal in the sense that the equality holds for special Shimura families (Viehweg-Zuo \cite{VZ2004}, M\"{o}ller-Viehweg-Zuo \cite{MVZ2006} see also \S \ref{section_exp_elliptic_family}). When $r=1$ and $f$ is a non-isotrivial semistable family of general type projective manifolds over a curve, it is proved by Lu-Yang-Zuo \cite{LYZ2022} that (\ref{align_main_Arakelov_ineq_H}) must be strict for those $k$ such that the $k$-th pluricanonical linear systems of the general fibers give rise to  birational maps.

	\subsection{Deformation boundedness of families of stable minimal models}
    Birkar \cite{Birkar2022} recently  introduces the moduli space of stable minimal models as a solution to the problem of constructing a compact moduli of birational equivalence classes of varieties of an arbitrary Kodaira dimension.
	The second main result of the present paper is a series of numerical inequalities of the Koll\'ar type polarizations of the moduli space of lc stable minimal models (Theorem \ref{thm_main_numbound_moduli}). 
	Let us first briefly review the main constructions in \cite{Birkar2022}. 
	Let
	$$d\in\bN,c\in\bQ^{\geq0},\Gamma\subset\bQ^{>0} \textrm{ a finite set, and }\sigma\in\bQ[t].$$
	A $(d,\Phi_c,\Gamma,\sigma)$-stable minimal model $(X,B),A$ consists of a reduced connected projective scheme $X$ which is of finite type over ${\rm Spec}(\bC)$ and $\bQ$-divisors $A\geq0$, $B\geq 0$ such that the following hold.
	\begin{itemize}
		\item $\dim X=d$, $(X,B)$ is a slc projective pair. $K_X+B$ is semi-ample,
		\item the coefficients of $A$ and $B$ are in $c\bZ^{\geq0}$,
		\item $(X,B+tA)$ is slc and $K_X+B+tA$ is ample for some $t>0$,
		\item ${\rm vol}(K_X+B+tA)=\sigma(t)$ for $0\leq t\ll 1$,
		\item ${\rm vol}(A|_F)\in\Gamma$ where $F$ is any general fiber of the fibration $f:X\to Z$ determined by $K_X+B$.
	\end{itemize}
	A $(d,\Phi_c,\Gamma,\sigma)$-stable minimal model $(X,B),A$ is called lc (resp. klt) when $(X,B)$ is a lc pair (resp. a klt pair).
	In \cite{Birkar2022}, Birkar shows that the families of $(d,\Phi_c,\Gamma,\sigma)$-stable minimal models (see \cite{Birkar2022} or \S \ref{section_moduli} for precise definitions) over reduced bases form a proper Deligne-Mumford stack $\sM_{\rm slc}(d,\Phi_c,\Gamma,\sigma)$ which admits a projective good coarse moduli space $M_{\rm slc}(d,\Phi_c,\Gamma,\sigma)$.
	
	$M_{\rm slc}(d,\Phi_c,\Gamma,\sigma)$ admits a set of polarizations (i.e. ample $\bQ$-line bundles) of Koll\'ar type (see \S \ref{section_polarization_moduli} for details). Let $a\in\bQ^{>0}$ be sufficiently small and let $r\in\bZ^{>0}$ be sufficiently large. Then the assignment
	$$f:(X,A)\to S\in\sM_{\rm slc}(d,\Phi_c,\Gamma,\sigma)(S)\mapsto f_\ast(r(K_{X/S}+aA))$$
    determines a locally free coherent sheaf on the stack $\sM_{\rm slc}(d,\Phi_c,\Gamma,\sigma)$, which is denoted by $\Lambda_{a,r}$. Let $\lambda_{a,r}:=\det(\Lambda_{a,r})$. Since $\sM_{\rm slc}(d,\Phi_c,\Gamma,\sigma)$ is Deligne-Mumford, some power $\lambda_{a,r}^{\otimes k}$ descends to a line bundle on $M_{\rm slc}(d,\Phi_c,\Gamma,\sigma)$. Therefore we regard $\lambda_{a,r}$ as a $\bQ$-line bundle on $M_{\rm slc}(d,\Phi_c,\Gamma,\sigma)$. These $\lambda_{a,r}$ are ample for $0<a\ll1$ and $r\gg0$ (Koll\'ar \cite{Kollar1990}, Kov\'acs-Patakfalvi \cite{Kovacs2017} and Fujino \cite{Fujino2018}). 
	
	In the present paper we investigate the deformation boundedness of the families of $(d,\Phi_c,\Gamma,\sigma)$-lc stable minimal models. However, the boundedness usually does not hold for the family of stable minimal models, according to the following examples. 
	\begin{exmp}[the presence of the degenerating fiber]\label{exmp_1}
		Let us consider a Lefschetz pencil $f:X\to\bP^1$ with $S$ the set of its critical values. We assume that the general fibers of $f$ are canonically polarized $d$-folds with $v$ their volumes so that they are $(d,\Phi_0,\{1\},v)$-stable minimal models ($\#(S)\geq3$ due to \cite{VZ2001}). Then $f$ is a family of $(d,\Phi_0,\{1\},v)$-stable minimal models. Let $\tau:\bP^1\to\bP^1$ be a morphism and denote $f_{\tau}:X\times_{\bP^1}\bP^1\to\bP^1$ to be the base change of $f$ via $\tau$. Since the number of degenerating fibers $\#(\tau^{-1}(S))$ can be arbitrarily large, the set of families $\{f_\tau|\tau:\bP^1\to\bP^1\}$ can not live in a bounded family of families of $(d,\Phi_0,0,v)$-stable minimal models over $\bP^1$.
	\end{exmp}
    \begin{exmp}[the presence of the degenerating polarization]\label{exmp_2}
    	Let $E$ be an elliptic curve and $x_0\in E(\bC)$. Denote $X=E\times E$ and let $f:E\times E\to E$ be the projection to the first component. Denote $A=\frac{1}{2}(E\times\{x_0\}\cup\Delta_E)$ where $\Delta_E\subset E\times E$ is the diagonal. Then $f:(X,A)\to E$ is a family of $(1,\Phi_{\frac{1}{2}},\{1\},t)$-stable minimal models. The underlying family of elliptic curves is trivial but the family of polarization is non-isotrivial (it degenerates at $x_0$). By taking the base change family $f_n$ via morphism $\times n:E\to E$ for various $n\geq 1$, we obtain families $\{f_n\}$ that can not live in a bounded family because the number of the degenerating loci $\{\frac{1}{n}x_0,\dots,\frac{(n-1)}{n}x_0\}$ could be arbitrarily large.
    \end{exmp}
	We will see that the presences of the degenerating fibers and the  degenerating polarizations are the only two obstructions of the deformation boundedness of families of lc stable minimal models. 
	\subsubsection{The second Arakelov inequality}
	A family $f:(X,B),A\to S$ of $(d,\Phi_c,\Gamma,\sigma)$-lc stable minimal model is called {\bf admissible} if it admits a log smooth birational model (Definition \ref{defn_admissible}) and the coefficients of $B$ lie in $(0,1)$. 
	The second Arakelov type inequality of the present paper is the following.
	\begin{thm}[Uniform numerical bound of the polarization, =Theorem \ref{thm_numbound_polarization}]\label{thm_main_numbound_moduli}
		Let $f^o:(X^o,B^o),A^o\to S^o$ be an admissible family of $(d,\Phi_c,\Gamma,\sigma)$-lc stable minimal models over a smooth quasi-projective variety $S^o$. 
		Let $S$ be a smooth projective variety containing $S^o$ as a Zariski open subset. Assume that $D:=S\backslash S^o$ is a divisor and the morphism $\xi^o:S^o\to M_{\rm lc}(d,\Phi_c,\Gamma,\sigma)$ induced from the family $f^o$ extends to a morphism $\xi:S\to M_{\rm slc}(d,\Phi_c,\Gamma,\sigma)$\footnote{We do not require $\xi$ to have moduli interpretation.}. Let $0<a\ll1$ and $1\ll r\in\bZ$. Assume that $K_{S}+D$ is pseudo-effective. Then the following inequalities hold.
		\begin{align*}
		c_1(\xi^\ast\lambda_{a,r})A_1A_2\cdots A_{\dim S-1}\leq \frac{rd{\rm rank}(\Lambda_{a,r})}{2}(K_S+D)A_1A_2\cdots A_{\dim S-1}
		\end{align*}
		for any semiample effective divisors $A_1,\dots, A_{\dim S-1}$ on $S$, and
		\begin{align*}
		c_1(\xi^\ast\lambda_{a,r})\cdot\alpha\leq rd{\rm rank}(\Lambda_{a,r})\left(K_S+D\right)\cdot\alpha+D\cdot\alpha
		\end{align*}
		for every movable class $\alpha\in N_1(S)$. If in particular  $\dim S=1$, then
		\begin{align}
		\deg(\xi^\ast\lambda_{a,r})\leq \frac{rd{\rm rank}(\Lambda_{a,r})}{2}\deg(K_S+D).
		\end{align}
	\end{thm}
	Readers may see Theorem \ref{thm_numbound_polarization} for the precise bounds of $a$ and $r$ so that the theorem is valid. Examples \ref{exmp_1} and \ref{exmp_2} show that the condition "admissible" is necessary for the inequalities.
	\subsubsection{Deformation boundedness of admissible families of  $(d,\Phi_c,\Gamma,\sigma)$-lc stable minimal models}\label{section_defbounded}
	 Theorem \ref{thm_main_numbound_moduli}, combined with the work of Kov\'acs-Lieblich \cite{Kovacs2011}, leads to the deformation boundedness of admissible families of $(d,\Phi_c,\Gamma,\sigma)$-lc stable minimal models. The $(d,\Phi_c,\Gamma,\sigma)$-lc stable minimal models $(X,B),A$ such that the coefficients of $B$ lie in $(0,1)$ form an open substack $\sM_{\rm lc,(0,1)}(d,\Phi_c,\Gamma,\sigma)\subset \sM_{\rm slc}(d,\Phi_c,\Gamma,\sigma)$ which admits a quasi-projective coarse moduli space $M_{\rm lc,(0,1)}(d,\Phi_c,\Gamma,\sigma)$.
	\begin{thm}\label{main_thm_bounded_stable_family}
		Let $S$ be an algebraic variety such that $S_{\rm sing}$ is compact. 
		Then there is a scheme of finite type ${\bf M}$ and a morphism $S\times {\bf M}\to M_{\rm lc,(0,1)}(d,\Phi_c,\Gamma,\sigma)$ that contains all maps $S\to M_{\rm lc,(0,1)}(d,\Phi_c,\Gamma,\sigma)$ which is induced from an admissible family of $(d,\Phi_c,\Gamma,\sigma)$-stable minimal models over $S$.
	\end{thm}
	We would like to remark that the morphism $S\times {\bf M}\to M_{\rm lc,(0,1)}(d,\Phi_c,\Gamma,\sigma)$ may not come from a family. Even if $S\times {\bf M}\to M_{\rm lc,(0,1)}(d,\Phi_c,\Gamma,\sigma)$ comes from a family $F:(\sX,A)\to S\times {\bf M}$ of $(d,\Phi_c,\Gamma,\sigma)$-stable minimal models, there may exist $p\in {\bf M}(\bC)$ such that $(\sX_p,A_p)\to S\times\{p\}$ is not admissible. 
	For admissible families of log smooth stable minimal models things are much better.
	A family $f:(X,B),A\to S$ of $(d,\Phi_c,\Gamma,\sigma)$-klt stable minimal models is called {\bf log smooth} if $X\to S$ is smooth and $A+B$ is an $f$-relative simple normal crossing $\bQ$-divisor.
	The groupoids of log smooth families of $(d,\Phi_c,\Gamma,\sigma)$-klt stable minimal models forms an open substack (denoted by $\sM_{\rm sm}(d,\Phi_c,\Gamma,\sigma)$) of $\sM_{\rm slc}(d,\Phi_c,\Gamma,\sigma)$. 
	\begin{thm}\label{main_thm_bounded_stable_family_strong}
		Let $S$ be an algebraic variety such that $S_{\rm sing}$ is compact. 
		Then there is a  scheme of finite type ${\bf M}$ and an admissible log smooth family $F\in\sM_{\rm sm}(d,\Phi_c,\Gamma,\sigma)(S\times {\bf M})$ of klt stable minimal models which contains all admissible log smooth families of $(d,\Phi_c,\Gamma,\sigma)$-klt stable minimal models over $S$.
	\end{thm} 
	\begin{exmp}[Log smooth families of projective pairs of log general type]
		Let $f:(X,B)\to S$ be a log smooth family of projective klt pairs of log general type. Assume that the fibers $(X_s,B_s)$ have dimension $d$, volume ${\rm vol}(K_{X_s}+B_s)=v$ and the coefficients of $B$ lie in $c\bZ^{\geq0}$ for some $c\in \bQ^{\geq0}$. Then the relative lc model $(X^{\rm can},B^{\rm can}),0\to S$ (c.f. \cite{BCHM2010}) is an admissible family of $(d,\Phi_c,\{1\},v)$-lc stable minimal models (see \cite[Page 721]{WeiWu2023}). Hence $f$ determines a morphism $S\to\sM_{\rm lc}(d,\Phi_c,\{1\},v)$. By Theorem \ref{main_thm_bounded_stable_family} we have the following claim.
		\begin{cor}
			Fix $d\in\bN,c\in\bQ^{\geq0}$ and $v\in\bQ^{>0}$.
			Let $S$ be an algebraic variety such that $S_{\rm sing}$ is compact. 
			Then there is a scheme of finite type  ${\bf M}$ and a morphism $S\times {\bf M}\to M_{\rm lc,(0,1)}(d,\Phi_c,\{1\},v)$ that contains all the map $S\to M_{\rm lc,(0,1)}(d,\Phi_c,\{1\},v)$ which is induced from a log smooth family of projective pairs of general type.
		\end{cor}
	    This generalizes the deformation boundedness of families of general type surfaces by Bedulev-Viehweg \cite{Viehweg2000} to arbitrary dimensions.
	\end{exmp}
    \begin{exmp}[Families of Calabi-Yau varieties]
    	A lc stable minimal model $(X,B),A$ is a stable Calabi-Yau pair if $K_X+B\sim_{\bQ}0$. Theorem \ref{main_thm_bounded_stable_family} implies the deformational boundedness of  $(d,\Phi_c,\Gamma,\sigma)$-lc stable Calabi-Yau pairs. Theorem \ref{main_thm_bounded_stable_family_strong} ensures that there is a family of finite type that parameterizes all log smooth $(d,\Phi_c,\Gamma,\sigma)$-klt stable Calabi-Yau pairs.
    \end{exmp}
    \begin{exmp}[Families of stable Fano pairs]
    	A lc stable minimal model $(X,B),A$ is a stable Fano pair if $(X,A+B),A$ is a stable Calabi-Yau pair. Theorem \ref{main_thm_bounded_stable_family} implies the deformational boundedness of $(d,\Phi_c,\Gamma,\sigma)$-lc stable Fano pairs.
    \end{exmp}
    \begin{exmp}[Families of marked curves]
    	Let $S$ be a smooth variety and let $(X,D)\to S$ be a family of smooth curves of genus $g$ with $m\geq0$ distinct marked points such that $2g-2+m>0$, i.e. $D$ is a smooth divisor with $m$ connected components such that each component is mapped isomorphically onto $S$. Then $(X,D),0\to S$ is a log smooth family of $(1,\Phi_1,\{m\},mt+2g-2)$-lc stable minimal models. Theorem \ref{main_thm_bounded_stable_family_strong} ensures that there is a family of finite type  $(\sX,\sD),0\to{\bf M}\times S$ that parameterizes all families $(X,D)\to S$ of smooth curves of genus $g$ with $m$ distinct marked points.
    \end{exmp}
	The present paper is organized as follows. Section 2 contains the preliminary results on the theory of degeneration and prolongations of a variation of Hodge structure. The main result of Section 2 is a comparison between the pushforward of the dualizing sheaf and the analytic prolongation of the variation of Hodge structure (Proposition \ref{prop_prolongation0_vs_geo}). In Section 3 we introduce the analytic prolongation of Viehweg-Zuo's Higgs sheaves and prove two meta Arakelov type inequalities. Theorem \ref{thm_main_Arakelov_inequality} is proved in Section 4.  In Section 5 we illustrate by an example how the Arakelov bound effects the geometry of the family. We investigate the deformation boundedness for family of admissible lc stable minimal models in Section 6.
	
	{\bf Notations:}
	\begin{itemize}
		\item All the complex spaces are assumed to be separated, reduced, paracompact, countable at infinity and of pure dimension.
		\item Let $X$ be a complex space and $Z\subset X$ a closed analytic subset containing the singular loci $X_{\rm sing}$. A desingularization (resp. functorial desingularization) of the pair $(X,Z)$ is a projective morphism $\pi:\widetilde{X}\to X$ such that $\widetilde{X}$ is smooth, $\pi$ is biholomorphic over $X\backslash Z$, $\pi^{-1}(Z)$ and the exceptional loci ${\rm Ex}(\pi)$ are simple normal crossing divisors on $\widetilde{X}$ (resp. which is functorial in the sense of W\l odarczyk \cite{Wlodarczyk2009}). Notice a functorial desingularization $\pi$ is biholomorphic over the largest open subset $U\subset X_{\rm reg}$ where $U\cap Z\subset U$ is a simple normal crossing divisor.
		\item Let $f:Y\to X$ be a proper holomorphic morphism from a complex space to a connected complex manifold. Let $X'\to X$ be a holomorphic morphism between complex manifolds. The main component of $X'\times_X Y$ is the union of irreducible components of $X'\times_X Y$ which is mapped onto $X'$.
		\item The co-support of a coherent ideal sheaf $I\subset\sO_X$ on a complex space is defined to be ${\rm supp}(\sO_X/I)$.
	\end{itemize}
	\section{Analytic Prolongation of variation of Hodge structure}
	\subsection{Norm estimate for the Hodge metric}
	Let $\bV=(\cV,\nabla,\cF^\bullet,Q)$ be an $\bR$-polarized variation of Hodge structure over $(\Delta^\ast)^n\times \Delta^m$ where $(\cV,\nabla)$ is a flat connection, $\cF^\bullet$ is the Hodge filtration and $Q$ is a real polarization. Let $h_Q$ denote the associated Hodge metric. Let $s_1,\dots,s_n$ be holomorphic coordinates on $(\Delta^\ast)^n$ and denote $D_i:=\{s_i=0\}\subset\Delta^{n+m}$. Let $N_i$ be the unipotent part of ${\rm Res}_{D_i}\nabla$ and let 
	$$p:\bH^{n}\times \Delta^m\to (\Delta^\ast)^n\times \Delta^m,$$ 
	$$(z_1,\dots,z_n,w_1,\dots,w_m)\mapsto(e^{2\pi\sqrt{-1}z_1},\dots,e^{2\pi\sqrt{-1}z_n},w_1,\dots,w_m)$$
	be the universal covering. Let
	$W^{(1)}=W(N_1),\dots,W^{(n)}=W(N_1+\cdots+N_n)$ be the monodromy weight filtrations (centered at 0) on $V:=\Gamma(\bH^n\times \Delta^m,p^\ast\cV)^{p^\ast\nabla}$.
	The following norm estimate for flat sections is proved by Cattani-Kaplan-Schmid \cite[Theorem 5.21]{Cattani_Kaplan_Schmid1986} for the case when $\bV$ has quasi-unipotent local monodromy and by Mochizuki \cite[Part 3, Chapter 13]{Mochizuki20072} for the general case.
	\begin{thm}\label{thm_Hodge_metric_asymptotic}
		For any $0\neq v\in {\rm Gr}_{l_n}^{W^{(n)}}\cdots{\rm Gr}_{l_1}^{W^{(1)}}V$, one has
		\begin{align*}
		|v|^2_{h_Q}\sim \left(\frac{\log|s_1|}{\log|s_2|}\right)^{l_1}\cdots\left(-\log|s_n|\right)^{l_n}
		\end{align*}
		over any region of the form
		$$\left\{(s_1,\dots, s_n,w_1,\dots,w_m)\in (\Delta^\ast)^n\times \Delta^m\bigg|\frac{\log|s_1|}{\log|s_2|}>\epsilon,\dots,-\log|s_n|>\epsilon,(w_1,\dots,w_m)\in K\right\}$$
		for any $\epsilon>0$ and an arbitrary compact subset $K\subset \Delta^m$ .
	\end{thm}
	Denote $S(\bV)=\cF^{\max\{p|\cF^p\neq0\}}$. The rest of this part is devoted to the norm estimate of $h_Q$ on $S(\bV)$. Denote $\cV_{-1}$ to be Deligne's canonical extension of $(\cV,\nabla)$ whose real parts of the eigenvalues of the residue maps lie in $(-1,0]$. By the nilpotent orbit theorem \cite{Cattani_Kaplan_Schmid1986} $j_\ast S(\bV)\cap\cV_{-1}$ is a subbundle of $\cV_{-1}$.
	\begin{lem}\label{lem_W_F}
		Assume that $n=1$. Then $W_{-1}(N_1)\cap \big(j_\ast S(\bV)\cap\cV_{-1}\big)_{\bf 0}=0$.
	\end{lem}
	\begin{proof}
		Assume that $W_{-1}(N_1)\cap \big(j_\ast S(\bV)\cap\cV_{-1}\big)_{\bf 0}\neq0$ and let $k$ be the weight of $\bV$. Let $l=\max\{l|W_{-l}(N_1)\cap \big(j_\ast S(\bV)\cap\cV_{-1}\big)_{\bf 0}\neq0\}$. Then $l\geq 1$. 
		By \cite[6.16]{Schmid1973}, the filtration $j_\ast \cF^\bullet\cap\cV_{-1}$ induces a pure Hodge structure of weight $m+k$ on $W_{m}(N_1)/W_{m-1}(N_1)$. Moreover 
		\begin{align}\label{align_hard_lef_N}
		N^l: W_{l}(N_1)/W_{l-1}(N_1)\to W_{-l}(N_1)/W_{-l-1}(N_1)
		\end{align}
		is an isomorphism of type $(-l,-l)$. Denote $S(\bV)=\cF^p$. By the definition of $l$, any nonzero element $\alpha\in W_{-l}(N_1)\cap \big(j_\ast S(\bV)\cap\cV_{-1}\big)_{\bf 0}$ induces a nonzero $[\alpha]\in W_{-l}(N_1)/W_{-l-1}(N_1)$ of Hodge type $(p,k-l-p)$. Since (\ref{align_hard_lef_N}) is an isomorphism, there is $\beta\in W_{l}(N_1)/W_{l-1}(N_1)$ of Hodge type $(p+l,k-p)$ such that $N^l(\beta)=[\alpha]$. However, $\beta=0$ since $\cF^{p+l}=0$. This contradicts to the fact that $[\alpha]\neq0$. Consequently, $W_{-1}(N_1)\cap \big(j_\ast S(\bV)\cap\cV_{-1}\big)_{\bf 0}$ must be zero.
	\end{proof}
	Let $T_i$ denote the local monodromy operator of $\bV$ around $D_i$.
	Since $T_1,\dots,T_n$ are pairwise commutative, there is a finite decomposition 
	$$\cV_{-1}|_{\bf 0}=\bigoplus_{-1<\alpha_1,\dots,\alpha_n\leq 0}\bV_{\alpha_1,\dots,\alpha_n}$$
	such that $(T_i-e^{2\pi\sqrt{-1}\alpha_i}{\rm Id})$ is unipotent on $\bV_{\alpha_1,\dots,\alpha_n}$ for each $i=1,\dots,n$. 
	Let $$v_1,\dots, v_N\in (\cV_{-1}\cap j_\ast S(\bV))|_{\bf 0}\cap\bigcup_{-1<\alpha_1,\dots,\alpha_n\leq 0}\bV_{\alpha_1,\dots,\alpha_n}$$
	be an orthogonal basis of $(\cV_{-1}\cap j_\ast S(\bV))|_{\bf 0}\simeq \Gamma(\bH^n\times\Delta^m,p^\ast S(\bV))^{p^\ast\nabla}$. Then $\widetilde{v_1},\dots,\widetilde{v_N}$ that are determined by
	\begin{align}\label{align_adapted_frame}
	\widetilde{v_j}:={\rm exp}\left(\sum_{i=1}^n\log z_i(\alpha_i{\rm Id}+N_i)\right)v_j\textrm{ if } v_j\in\bV_{\alpha_1,\dots, \alpha_n},\quad \forall j=1,\dots,N
	\end{align}
	form a frame of $\cV_{-1}\cap j_\ast S(\bV)$.
	We always use the notation $\alpha_{D_i}(\widetilde{v_j})$ instead of $\alpha_i$ in (\ref{align_adapted_frame}). By (\ref{align_adapted_frame}) we see that 
	\begin{align*}
	|\widetilde{v_j}|^2_{h_Q}&\sim\left|\prod_{i=1}^nz_i^{\alpha_{D_i}(\widetilde{v_j})}{\rm exp}\left(\sum_{i=1}^nN_i\log z_i\right)v_j\right|^2_{h_Q}\\\nonumber
	&\sim|v_j|^2_{h_Q}\prod_{i=1}^n |z_i|^{2\alpha_{D_i}(\widetilde{v_j})},\quad j=1,\dots,N
	\end{align*}
	where $\alpha_{D_i}(\widetilde{v_j})\in(-1,0]$, $\forall i=1,\dots, n$. 
	It follows from Theorem \ref{thm_Hodge_metric_asymptotic} and Lemma \ref{lem_W_F} that 
	\begin{align*}
	|v_j|^2_{h_Q}\sim \left(\frac{\log|s_1|}{\log|s_2|}\right)^{l_1}\cdots\left(-\log|s_n|\right)^{l_n},\quad l_1\leq l_2\leq\dots\leq l_{n},
	\end{align*}
	over any region of the form
	$$\left\{(s_1,\dots, s_n,w_1,\dots,w_{m})\in (\Delta^\ast)^n\times \Delta^{m}\bigg|\frac{\log|s_1|}{\log|s_2|}>\epsilon,\dots,-\log|s_n|>\epsilon,(w_1,\dots,w_{m})\in K\right\}$$
	for any $\epsilon>0$ and an arbitrary compact subset $K\subset \Delta^{m}$. Hence 
	\begin{align*}
	1\lesssim |v_j|\lesssim|z_1\cdots z_n|^{-\epsilon},\quad\forall\epsilon>0.
	\end{align*}
	The local frame $(\widetilde{v_1},\dots,\widetilde{v_N})$ is $L^2$-adapted in the following sense.
	\begin{defn}(S. Zucker \cite[page 433]{Zucker1979})
		Let $(E,h)$ be a vector bundle with a possibly singular hermitian metric $h$ on a hermitian manifold $(X,ds^2_0)$. A holomorphic local frame $(v_1,\dots,v_N)$ of $E$ is called $L^2$-adapted if, for every set of measurable functions $\{f_1,\dots,f_N\}$, $\sum_{i=1}^Nf_iv_i$ is locally square integrable if and only if $f_iv_i$ is locally square integrable for each $i=1,\dots,N$.
	\end{defn}
	To see that $(\widetilde{v_1},\dots,\widetilde{v_N})$ is $L^2$-adapted, let us consider the measurable functions $f_1,\dots,f_N$. If 
	$$\sum_{j=1}^N f_j\widetilde{v_j}={\rm exp}\left(\sum_{i=1}^nN_i\log z_i\right)\left(\sum_{j=1}^N f_j\prod_{i=1}^n |z_i|^{\alpha_{D_i}(\widetilde{v_j})}v_j\right)$$
	is locally square integrable, then 
	$$\sum_{j=1}^N f_j\prod_{i=1}^n |z_i|^{\alpha_{D_i}(\widetilde{v_j})}v_j$$
	is locally square integrable because the entries of the matrix ${\rm exp}\left(-\sum_{i=1}^nN_i\log z_i\right)$ are $L^\infty$-bounded.
	Since $(v_1,\dots,v_N)$ is an orthogonal basis, 
	$|f_j\widetilde{v_j}|_{h_Q}\sim\prod_{i=1}^n |z_i|^{\alpha_{D_i}(\widetilde{v_j})}|f_jv_j|_{h_Q}$ is locally square integrable for each $j=1,\dots,N$. 
	
	In conclusion, we obtain the following proposition.
	\begin{prop}\label{prop_adapted_frame}
		Let $(X,ds^2_0)$ be a hermitian manifold and $D$ a normal crossing divisor on $X$. Let $\bV$ be an $\bR$-polarized variation of Hodge structure on $X^o:=X\backslash D$. Then there is an $L^2$-adapted holomorphic local frame $(\widetilde{v_1},\dots,\widetilde{v_N})$ of $\cV_{-1}\cap j_\ast S(\bV)$ at every point $x\in D$. Let $z_1,\cdots,z_n$ be holomorphic local coordinates on $X$ so that $D
		=\{z_1\cdots z_r=0\}$. Then there are $\alpha_{D_i}(\widetilde{v_j})\in(-1,0]$, $i=1,\dots, r$, $j=1,\dots,N$ and positive real functions $\lambda_j\in C^\infty(X\backslash D)$, $j=1,\dots,N$ such that
		\begin{align}\label{align_L2adapted_frame}
		|\widetilde{v_j}|^2\sim\lambda_j\prod_{i=1}^r |z_i|^{2\alpha_{D_i}(\widetilde{v_j})},\quad \forall j=1,\dots,N
		\end{align}
		and
		$$1\lesssim \lambda_j\lesssim|z_1\cdots z_r|^{-\epsilon},\quad\forall\epsilon>0$$
		for each $j=1,\dots,N$.
	\end{prop}
	\subsection{Prolongation of a VHS: log smooth case}\label{section_prolongation}
	Let $X$ be a complex manifold and $D=\sum_{i=1}^l D_i$ a reduced simple normal crossing divisor on $X$.
	Let $(E,h)$ be a holomorphic vector bundle on $X\backslash D$ with a smooth hermitian metric $h$. Let $D_1=\sum_{i=1}^la_iD_i$, $D_2=\sum_{i=1}^lb_iD_i$ be $\bR$-divisors. We denote $D_1<(\leq) D_2$ if $a_i<(\leq) b_i$ for all $i=1,\dots,l$.
	\begin{defn}[Prolongation](Mochizuki \cite{Mochizuki2002}, Definition 4.2)\label{defn_prolongation}
		Let $A=\sum_{i=1}^la_iD_i$ be an $\bR$-divisor, let $U$ be an open subset of $X$, and let $s\in \Gamma(U\backslash D,E)$ be a holomorphic section. We denote $(s)\leq -A$ if $|s|_h=O(\prod_{k=1}^r |z_{k}|^{-a_{i_k}-\epsilon})$ for any positive number $\epsilon$, where $z_1,\dots,z_n$ are holomorphic local coordinates such that $D=\{z_1\cdots z_r=0\}$ and $D_{i_k}=\{z_k=0\}$, $k=1,\dots,r$.
		The $\sO_X$-module $ _{A}E$ is defined as 
		$$\Gamma(U, {_{A}}E):=\{s\in\Gamma(U\backslash D,E)|(s)\leq -A\}$$
		for any open subset $U\subset X$.
		Denote
		\begin{align*}
		{_{<A}}E:=\bigcup_{{B}<{A}}{_{B}}E\quad\textrm{and}\quad{\rm Gr}_{A}E:={_{A}}E/{_{<A}}E.
		\end{align*}	
	\end{defn}
	Let $\bV=(\cV,\nabla,\cF^\bullet,Q)$ be an $\bR$-polarized variation of Hodge structure of weight $w$ on $X\backslash D$. Let $(H:={\rm Gr}_{\cF^\bullet}\cV,\theta:={\rm Gr}_{\cF^\bullet}\nabla)$ denote the total graded quotient. Then $(H,\theta)$ is the Higgs bundle  corresponding to $(\cV,\nabla)$ via Simpson's correspondence \cite{Simpson1988}. The Hodge metric $h_Q$ associated with $Q$ is a harmonic metric on $(H,\theta)$. The triple $(H,\theta,h_Q)$ is a tame harmonic bundle in the sense of Simpson \cite{Simpson1990} and Mochizuki \cite{Mochizuki20072}. Notice that $(H,\theta)$ is a system of Hodge bundles (Simpson \cite[\S 4]{Simpson1992}) in the sense that
	\begin{align*}
	H=\bigoplus_{p+q=w} H^{p,q},\quad H^{p,q}\simeq \cF^{p}/\cF^{p+1},\quad \theta(H^{p,q})\subset H^{p-1,q+1}.
	\end{align*}
	According to Simpson \cite[Theorem 3]{Simpson1990} and Mochizuki \cite[Proposition 2.53]{Mochizuki2009},  the prolongations forms a parabolic structure as follows.
	\begin{thm}\label{thm_parabolic}
		Let $X$ be a complex manifold and $D=\sum_{i=1}^l D_i\subset X$ a reduced simple normal crossing divisor. Let $(H=\oplus_{p+q=w} H^{p,q},\theta,h_Q)$ be the system of Hodge bundles associated with an $\bR$-polarized variation of Hodge structure of weight $w$ on $X\backslash D$.
		For each $\bR$-divisor $A$ supported on $D$, $_{A}H$ is a locally free coherent sheaf such that the following hold.
		\begin{itemize}
			\item  $_{A+\epsilon D_i}H = {_A}H$ for any $i=1,\dots,l$ and any constant $0<\epsilon\ll 1$.
			\item $_{A+D_i}H={_A}H\otimes \sO(-D_i)$ for every $1\leq i\leq l$. 
			\item The subset of $(a_1,\dots,a_l)\in\bR^l$ such that ${\rm Gr}_{\sum_{i=1}^l a_iD_i}H\neq 0$ is discrete. 
			\item The Higgs field $\theta$ has at most logarithmic poles along $D$, i.e. $\theta$ extends to 
			\begin{align*}
			{_{A}}H\to{_{A}}H\otimes\Omega_{X}(\log D).
			\end{align*}
		\end{itemize} 
	\end{thm}
	The proof of the following lemma is straightforward. Thus we omit it here.
	\begin{lem}\label{lem_integral}
		Let $f$ be a holomorphic function on $\Delta^\ast:=\{z\in\bC|0<|z|<1\}$ and $a\in\bR$. Then
		$$\int_{|z|\leq\frac{1}{2}}|f|^2|z|^{2a}dzd\bar{z}<\infty$$
		if and only if $v(f)+a>-1$. Here
		$$v(f):=\min\{l|f_l\neq0\textrm{ in the Laurent expansion } f=\sum_{i\in\bZ}f_iz^i\}.$$
	\end{lem}
	\begin{lem}\label{lem_Deligne_prolongation}
		Notations as above. Let $S(\bV)=\cF^{\max\{p|\cF^p\neq0\}}$. Then one has a natural isomorphism $$\cV_{-1}\cap j_\ast S(\bV)\simeq {_{<D}}S(\bV).$$
		Here $j:X\backslash D\to X$ is the immersion and ${_{<D}}S(\bV)$ is taken with respect to the Hodge metric $h_Q$. Let $U\subset X$ be an open subset. Then a holomorphic section $s\in S(\bV)(U\backslash D)$ extends to a section in ${_{<D}}S(\bV)(U)$ if and only if it is locally square integrable at every point of $U\cap D$. That is, the integration
		$$\int|s|^2_{h_Q}{\rm vol}_{ds^2}$$
		is finite locally at every point of $U\cap D$, where $ds^2$ is a hermitian metric on $X$.
	\end{lem}
	\begin{proof}
		It follows from Proposition \ref{prop_adapted_frame} and Lemma \ref{lem_integral} that 
		$$\cV_{-1}\cap j_\ast S(\bV)\subset {_{<D}}S(\bV).$$
		For the converse, let $\widetilde{v_1},\dots,\widetilde{v_N}$ be the $L^2$-adapted local frame of $\cV_{-1}\cap j_\ast S(\bV)$ at some point $x\in D$. Let $\alpha=\sum_{j=1}^N f_j\widetilde{v_j}\in {_{<D}}S(\bV)$ where $f_1,\dots,f_N$ are holomorphic outside $D$. By Lemma \ref{lem_integral}, $\alpha$ is locally square integrable at $x$. Hence all $f_j\widetilde{v_j}$ are locally square integrable at $x$ because $(\widetilde{v_1},\dots,\widetilde{v_N})$ are $L^2$-adapted. By (\ref{align_L2adapted_frame}) and Lemma \ref{lem_integral} again one knows that $f_1,\dots,f_N$ are holomorphic on some neighborhood of $x$. This proves
		$${_{<D}}S(\bV)\subset \cV_{-1}\cap j_\ast S(\bV)$$
		and the last claim of the lemma.
	\end{proof}
	\subsection{Prolongation of a VHS: general case}\label{section_prolongation_general} 
	The analytic prolongation of a variation of Hodge structure over a general base is defined via desingularization. 
	Let $X$ be a complex manifold and $Z\subset X$ a closed analytic subset. Let $D\subset Z$ be the union of irreducible components of $Z$ whose codimension is one. Let $\pi:\widetilde{X}\to X$ be a functorial desingularization of the pair $(X,Z)$ (c.f. \cite{Wlodarczyk2009}) so that $\widetilde{X}$ is smooth, $\pi^{-1}(Z)$ is a simple normal crossing divisor on $\widetilde{X}$ and 
	$$\pi^o:=\pi|_{\widetilde{X}^o}:\widetilde{X}^o:=\pi^{-1}(X\backslash Z)\to X^o:=X\backslash Z$$
	is biholomorphic.
	Let $\bV=(\cV,\nabla,\cF^\bullet,Q)$ be an $\bR$-polarized variation of Hodge structure of weight $w$ on $\widetilde{X}^o$ and $(H=\oplus_{p+q=w}H^{p,q},\theta,h_Q)$ the corresponding Higgs bundle with the Hodge metric $h_Q$. Let $A$ be an $\bR$-divisor supported on $\pi^{-1}(Z)$. Then $\pi_\ast({_{A}}H)$ is a torsion free coherent sheaf on $X$ whose restriction on $X^o$ is $(\pi^o)^{-1\ast}(H)$. By abuse of notation we still denote $\theta:=(\pi^o)^{-1\ast}(\theta)$. $\theta$ is a meromorphic Higgs field on $\pi_\ast({_{A}}H)$ with poles along $Z$. Let ${\rm Cryt}(\pi)\subset X$ be the degenerate loci of $\pi$. Since $\pi$ is functorial, $D\backslash {\rm Cryt}(\pi)$ is a simple normal crossing divisor on $X\backslash {\rm Cryt}(\pi)$ and the exceptional loci $\pi^{-1}({\rm Cryt}(\pi))$ is a simple normal crossing divisor on $\widetilde{X}$. $(\pi_\ast({_{A}}H),\theta)|_{X\backslash {\rm Cryt}(\pi)}$ is locally free and $\theta$ admits at most log poles along $D\backslash {\rm Cryt}(\pi)$. The following negativity result for the kernel of a Higgs field generalizes \cite{Brunebarbe2017} and \cite{Zuo2000}. The main idea of its proof is due to Brunebarbe \cite{Brunebarbe2017}.
	\begin{prop}\label{prop_semipositive_kernel}
		Notations as above. Assume that ${\rm supp}(A)$ lies in the exceptional divisor $\pi^{-1}({\rm Cryt}(\pi))$. Let $K\subset \pi_\ast({_{A}}H)$ be a coherent subsheaf such that $\theta(K)=0$. Then $K^\vee$ is weakly positive in the sense of Viehweg \cite{Viehweg1983}.
	\end{prop}
	\begin{proof}
		The Hodge metric $h_Q$ defines a singular hermitian metric on the bundle $_{A}H$, with singularities along $\pi^{-1}(Z)$, in accordance with the notion of singular hermitian metrics on torsion free coherent sheaves as discussed in \cite{Paun2018,HPS2018}.
		Since $\pi^o$ is biholomorphic, we may regard $h_Q$ as a singular hermitian metric on the torsion free coherent sheaf $\pi_\ast({_{A}}H)$. Denote  $K^o:=K|_{X^o}$. By Griffiths' curvature formula
		$$\Theta_{h_Q}(H)+\theta\wedge\overline{\theta}+\overline{\theta}\wedge\theta=0,$$
		one knows that
		$$\Theta_{h_Q}(K^o)=-\theta\wedge\overline{\theta}|_{K^o}+\overline{B}\wedge B$$
		is Griffiths semi-negative, where $B\in A^{1,0}_{X^o}(K^o,K^{o\bot})$ is the second fundamental class. We claim that the hermitian metric $h_Q|_{X^o}$ extends to a singular hermitian metric on $K$ with semi-negative curvatures. It suffices to prove the following assertion: Let $s\in K$ be a section. Then $\log|s|_{h_Q}$ extends to a plurisubharmonic function on $X$.
		
		Since $\Theta_{h_Q}(K^o)$ is Griffiths semi-negative, $\log|s|_{h_Q}$ is a smooth plurisubharmonic function on $X^o$. By  Riemannian extension theorem and Hartogs extension theorem for plurisubharmonic functions \cite[Lemma 12.4]{HPS2018} it suffices to show that $\log|s|_{h_Q}$ is locally bounded from above in codimension one. Let ${\rm Cryt}(\pi)\subset X$ be the degenerate loci of $\pi$, which is of codimension $\geq 2$. Then $D\backslash {\rm Cryt}(\pi)$ is a simple normal crossing divisor on $X\backslash {\rm Cryt}(\pi)$. By the assumption on $A$ one knows that $\pi_\ast({_{A}}H)|_{X\backslash {\rm Cryt}(\pi)}\simeq{_{\bm{0}}}((\pi^{o})^{-1\ast}H)$ where $\bm{0}$ is the zero divisor on $X\backslash {\rm Cryt}(\pi)$. Let $x$ be a general point of some component $D_i$ of $D\backslash {\rm Cryt}(\pi)$. Let $N_i$ be the monodromy operator along $D_i$ associated with the connection $((\pi^{o})^{-1\ast}\cV,(\pi^{o})^{-1\ast}\nabla)$ and let $\{W_k\}_{k\in\bZ}$ be the monodromy weight filtration determined by $N_i$. Since $\theta(s)=0$, one has $s\in W_0$ thanks to \cite[Corollary 6.7]{Schmid1973} (see also \cite[Lemma 5.4]{Brunebarbe2017}). Combining it with the fact that $s\in {_{\bm{0}}}((\pi^{o})^{-1\ast}H)$, Simpson's norm estimate \cite[page 721]{Simpson1990} implies that $|s|_{h_Q}$ is locally bounded at $x$. This implies the claim that $h_Q$ extends (uniquely) to a singular hermitian metric on $K$ with semi-negative curvature. Hence $K^{\vee}$ is weakly positive in the sense of Viehweg by \cite[Theorem 2.5.2]{PT2018}.
	\end{proof}
	\subsection{Prolongation of a VHS: K\"ahler families}
	Let $f:Y\to X$ be a proper morphism between complex manifolds and denote  $n:=\dim X-\dim Y$. Let $Z\subset X$ be a closed analytic subset such that $f$ is a K\"ahler submersion over $X^o:=X\backslash Z$. Denote $Y^o:=f^{-1}(X^o)$ and $f^o:=f|_{Y^o}:Y^o\to X^o$. Then $R^nf^o_\ast(\bR_{Y^o})$ underlies an $\bR$-polarized variation of Hodge structure $\bV^n_{f^o}=(\cV^n,\nabla,\cF^\bullet,Q)$ of weight $n$. Here $\cV^n\simeq R^nf^o_\ast(\bR_{Y^o})\otimes_{\bR}\sO_{X^o}$, $\nabla$ is the Gauss-Manin connection, $\cF^p\simeq R^nf^o_\ast(\Omega^{\geq p}_{Y^o/X^o})$ and $Q$ is an $\bR$-polarization associated with a $f^o$-relative K\"ahler form. Denote $h_Q$ to be the Hodge metric associated with $Q$.
	Let $(H^n_{f^o}=\oplus_{p+q=n}H^{p,q}_{f^o},\theta)$ be the Higgs bundle associated with $\bV^n_{f^o}$. Here $H^{p,q}_{f^o}\simeq R^qf^o_\ast(\Omega^p_{Y^o/X^o})$. 
	\begin{lem}\label{lem_prolongation0_vs_geo}
		Notations as above. Assume that $Z$ is a (reduced) simple normal crossing divisor. Then there is an isomorphism $$f_\ast(\omega_{Y/X})\simeq {_{<Z}}H^{n,0}_{f^o}.$$
	\end{lem}
	\begin{proof}
		Let $j:X^o\to X$ denote the open immersion.
		It suffices to show that
		${_{<Z}}H^{n,0}_{f^o}\otimes\omega_X=f_\ast(\omega_{Y})$ as subsheaves of $j_\ast H^{n,0}_{f^o}\otimes\omega_X$. Let $s\in H^{n,0}_{f^o}=f^o_\ast(\omega_{Y^o/X^o})$ be a holomorphic section. Denote $\phi=dz_1\wedge\cdots\wedge dz_{d}$ where $z_1,\dots,z_d$ are holomorphic local coordinates on $X$. Thanks to Lemma \ref{lem_Deligne_prolongation}, $s\in {_{<Z}}H^{n,0}_{f^o}$ if and only if the integral
		$$\int_{X^o}|s|^2_{h_Q}\phi\wedge\overline{\phi}=\epsilon_n\int_{X^o}\left(\int_{f^{-1}\{x\}}s(x)\wedge\overline{s(x)}\right)\phi\wedge\overline{\phi}=\epsilon_n\int_{Y^o}(s\wedge f^{o\ast}(\phi))\wedge\overline{s\wedge f^{o\ast}(\phi)}$$
		is finite locally at every point of $Z$, where $\epsilon_n=(-1)^{\frac{n(n-1)}{2}}(\sqrt{-1})^n$.
		The locally finiteness of the right handside is equivalent to that $s\wedge f^{o\ast}(\phi)$ admits a holomorphic extension to $Y$ (c.f. \cite[Proposition 16]{Kawamata1981}). This proves that  ${_{<Z}}H^{n,0}_{f^o}\otimes\omega_X=f_\ast(\omega_{Y})$.
	\end{proof}
	Let us return to the general case. Consider the diagram
	\begin{align}
	\xymatrix{
		Y'\ar[r]^{\sigma} \ar[d]^{f'} & Y \ar[d]^f\\
		X'\ar[r]^{\pi} &X
	}
	\end{align}
	such that the following hold.
	\begin{itemize}
		\item $\pi:X'\to X$ is a desingularization of the pair $(X,Z)$. In particular, $X'$ is smooth, $\pi^{-1}(Z)$ is a simple normal crossing divisor and $\pi^o:=\pi|_{\pi^{-1}(X^o)}:\pi^{-1}(X^o)\to X^o$ is biholomorphic.
		\item $Y'$ is a functorial desingularization of the main component of $Y\times_XX'$. In particular, $Y'\to Y\times_XX'$ is biholomorphic over $Y^o\times_{X^o}\pi^{-1}(X^o)$.
	\end{itemize}
	Let $\omega_{X'}\simeq\pi^{\ast}\omega_X\otimes\sO_{X'}(E)$
	for some exceptional divisor $E$ of $\pi$. We obtain natural morphisms
	\begin{align}\label{align_pullback_pushforward0}
	\pi^\ast(f_\ast(\omega_{Y/X}))\simeq\pi^\ast(f_\ast(\omega_Y)\otimes\omega_X^{-1})\to f'_\ast(\omega_{Y'})\otimes\omega^{-1}_{X'}\otimes\sO_{X'}(E)\simeq f'_\ast(\omega_{Y'/X'})\otimes\sO_{X'}(E).
	\end{align}
	Since $\pi^o$ is biholomorphic, $f'^o:\sigma^{-1}(Y^o)=(\pi f')^{-1}(X^o)\to\pi^{-1}(X^o)$ is a proper K\"ahler submersion. Let  $H^n_{f'^o}$ denote the Higgs bundle associated with $f'^o$. According to Lemma \ref{lem_prolongation0_vs_geo} we have
	\begin{align*}
	f'_\ast(\omega_{Y'/X'})\simeq {_{<\pi^{-1}(Z)_{\rm red}}}H^{n,0}_{f'^o}.
	\end{align*}
	Combining it with (\ref{align_pullback_pushforward0}) we obtain a generically injective morphism
	$$f_\ast(\omega_{Y/X})\to \pi_\ast({_{<\pi^{-1}(Z)_{\rm red}}}H^{n,0}_{f'^o}\otimes \sO_{X'}(E))\simeq \pi_\ast({_{<\pi^{-1}(Z)_{\rm red}+E}}H^{n,0}_{f'^o}).$$
	Since $f_\ast(\omega_{Y/X})$ is torsion free, this map has to be injective. Thus we prove the following.
	\begin{prop}\label{prop_prolongation0_vs_geo}
		Notations as above. There is an inclusion $$f_\ast(\omega_{Y/X})\subset \pi_\ast({_{<\pi^{-1}(Z)_{\rm red}+E}}H^{n,0}_{f'^o}).$$
	\end{prop}
	\section{Analytic prolongation of Viehweg-Zuo's Higgs sheaf}\label{section_Analytic_prolongation_VZ}
	This section is dedicated to the preparation for the proofs of  Theorem \ref{thm_main_Arakelov_inequality} and Theorem \ref{thm_main_numbound_moduli}. After introducing a technical but flexible geometrical setting, we generalize Viehweg-Zuo's construction of two Higgs sheaves to this setting by using analytic prolongations. An Arakelov type inequality under this setting is proved in Section \ref{section_AI_abstract}. 
	\subsection{Setting}\label{section_setting}
	Throughout this section let us fix a proper holomorphic morphism $f:Y\to X$ between complex manifolds  with  $n:=\dim Y-\dim X$ the relative dimension. We do not require $f$ to have connected fibers. Assume that there is a simple normal crossing divisor $D_f\subset X$ such that $f^o:=f|_{Y^o}:Y^o\to X^o$ is a K\"ahler submersion where $X^o:=X\backslash D_f$ and $Y^o:=f^{-1}(X^o)$. We fix a torsion free coherent sheaf $L$ on $X$ which is invertible on $X^o$ (hence ${\rm rank}(L)=1$), and a nonzero morphism
	\begin{align}\label{align_s}
	s_L:L^{\otimes k}\to f_\ast(\omega_{Y/X}^{\otimes k})
	\end{align}
	for some $k\geq 1$. 
	\subsection{Viehweg-Zuo's Higgs sheaves}\label{section_VZ_sheaf}
	Notations as in Section \ref{section_setting}. Let $L^{\vee\vee}$ be the reflexive hull of $L$ with $L\to L^{\vee\vee}$ the natural inclusion map. Since ${\rm rank}(L^{\vee\vee})=1$, $L^{\vee\vee}$ is an invertible sheaf. Because $L$ is torsion free and invertible on $X^o$, $\sI_T:=L\otimes(L^{\vee\vee})^{-1}\subset\sO_X$ is a coherent ideal sheaf whose co-support lies in a closed analytic subset $T\subset D_f$ such that ${\rm codim}_X(T)\geq 2$. Consider a diagram
	\begin{align}\label{align_setting}
	\xymatrix{
		\widetilde{Y}\ar[r]^{\sigma}\ar[d]_{\tilde{f}} & Y\ar[d]^f\\
		\widetilde{X}\ar[r]^{\pi} & X
	}
	\end{align}
	of holomorphic maps between complex manifolds such that the following hold. 
	\begin{itemize}
		\item $\pi$ is a functorial desingularization of $(X,T,D_f)$ in the sense of W{\l}odarczyk \cite{Wlodarczyk2009}. In particular,  $\widetilde{X}$ is a compact complex manifold, $\pi$ is a projective morphism which is biholomorphic over $X\backslash T$. $\pi^{-1}(D_f)$ and $E:=\pi^{-1}(T)$ are simple normal crossing divisors. 
		\item $\widetilde{Y}$ is a functorial desingularization of the main component of $Y\times_X\widetilde{X}$. In particular,  $\widetilde{Y}\to Y\times_X\widetilde{X}$ is biholomorphic over $f^{-1}(X\backslash T)\times_{X\backslash T}\pi^{-1}(X\backslash T)$.
	\end{itemize}
	Since $\pi$ is biholomorphic on $\widetilde{X}\backslash E$, there is a constant $k_0\geq 0$ and a natural map
	\begin{align*}
	\pi^{\ast}f_\ast(\omega_{Y/X}^{\otimes k})\otimes\sO_{\widetilde{X}}(-k_0kE)\to \widetilde{f}_\ast(\omega^{\otimes k}_{\widetilde{Y}/\widetilde{X}}).
	\end{align*}
	Taking (\ref{align_s}) into account, we obtain that
	\begin{align*}
	\pi^{\ast}(L^{\vee\vee})^{\otimes k}\otimes\pi^\ast(\sI_T)^{\otimes k}\simeq\pi^{\ast}L^{\otimes k}\to \pi^\ast(f_\ast(\omega_{Y/X}^{\otimes k}))\to \widetilde{f}_\ast(\omega^{\otimes k}_{\widetilde{Y}/\widetilde{X}})\otimes\sO_{\widetilde{X}}(k_0kE). 
	\end{align*}
	Hence there is an effective divisor $\widetilde{E}$, supported on $E$, such that there is a nonzero map
	\begin{align*}
	\pi^{\ast}(L^{\vee\vee})^{\otimes k}\otimes\sO_{\widetilde{X}}(-k\widetilde{E})\to  \widetilde{f}_\ast(\omega^{\otimes k}_{\widetilde{Y}/\widetilde{X}}). 
	\end{align*}
	Denote $\widetilde{L}:=\pi^{\ast}(L^{\vee\vee})\otimes\sO_{\widetilde{X}}(-\widetilde{E})$ and $L^o:=L|_{X^o}$. Denote $\pi^o:=\pi|_{\pi^{-1}(X^o)}:\pi^{-1}(X^o)\to X^o$. The arguments above show that there is a morphism
	\begin{align}
	s_{\widetilde{L}}:\widetilde{L}^{\otimes k}\to \widetilde{f}_\ast(\omega^{\otimes k}_{\widetilde{Y}/\widetilde{X}})
	\end{align} 
	and an isomorphism
	\begin{align}\label{align_iso_Lo_L}
	\pi^{o\ast}(L^o)\simeq\widetilde{L}|_{\pi^{-1}(X^o)}
	\end{align} 
	such that the diagram 
	\begin{align}\label{align_commut_s_s0}
	\xymatrix{
		\widetilde{L}^{\otimes k}|_{\pi^{-1}(X^o)}\ar[r]^-{s_{\widetilde{L}}}& \widetilde{f}_\ast(\omega^{\otimes k}_{\widetilde{Y}/\widetilde{X}})|_{\pi^{-1}(X^o)}
		\\
		\pi^{o\ast}(L^o)^{\otimes k}\ar[u]^{\simeq}\ar[r]^-{\pi^{o\ast}(s_L|_{X^o})} & \pi^{o\ast}f^o_\ast(\omega^{\otimes k}_{Y^o/X^o})\ar[u]^{\simeq}
	}
	\end{align}
	is commutative. 
	Define a line bundle
	$$B^o=\omega_{Y^o/X^o}\otimes f^{o\ast}(L^{o})^{-1}$$
	on $Y^o$ and a line bundle
	$$\widetilde{B}=\omega_{\widetilde{Y}/\widetilde{X}}\otimes \tilde{f}^\ast(\widetilde{L}^{-1})$$
	on $\widetilde{Y}$. Then the map $s_{\widetilde{L}}$ determines a non-zero section
	$\widetilde{s}\in H^0(\widetilde{Y},\widetilde{B}^{\otimes k})$.
	Let $\varpi:\widetilde{Y}_k\to \widetilde{Y}$ be the $k:1$ cyclic covering map branched along $\{\widetilde{s}=0\}$ and let $\mu:Z\to \widetilde{Y}_k$ be a functorial desingularization which is biholomorphic over the complement of $\{\varpi^\ast\widetilde{s}=0\}$. Denote  $g:=\widetilde{f}\varpi\mu$. The morphisms are gathered in the following diagram.
	\begin{align*}
	\xymatrix{
		Z\ar[r]^{\mu}\ar[drr]^g & \widetilde{Y}_k \ar[r]^{\varpi} & \widetilde{Y}\ar[d]^{\widetilde{f}}\ar[r]^{\sigma}& Y\ar[d]^f\\
		&& \widetilde{X}\ar[r]^\pi & X
	}.
	\end{align*}
	Let $D_g\subset\widetilde{X}$ be a reduced closed analytic subset containing $\pi^{-1}(D_f)$, such that $g$ is a submersion over $\widetilde{X}^o:=\widetilde{X}\backslash D_g$. Denote $Z^o:=g^{-1}(\widetilde{X}^o)$ and denote $g^o:=g|_{Z^o}:Z^o\to \widetilde{X}^o$ to be the restriction map. Since $\mu$ and $\varpi$ are projective morphisms, $g^o$ is a proper K\"ahler submersion. 
	
	Consider the diagram
	\begin{align}
	\xymatrix{
		Z'\ar[r]_{\sigma'} \ar[d]^{h} \ar@/^/[rr]^{\varphi}& Z \ar[d]^g\ar[r]_{\sigma\varpi\mu}& Y\ar[d]^f\\
		X'\ar[r]^{\rho} \ar@/_/[rr]_{\psi}&\widetilde{X}\ar[r]^\pi &X
	}
	\end{align}
	where $\varphi:=\sigma\varpi\mu\sigma'$ and $\psi:=\pi\rho$, such that the following hold.
	\begin{itemize}
		\item $\rho:X'\to \widetilde{X}$ is a functorial desingularization of the pair $(\widetilde{X},D_g)$. In particular, $X'$ is smooth, $\rho^{-1}(D_g)$ and $\psi^{-1}(D_f)$ are simple normal crossing divisors and $\rho^o:=\rho|_{\rho^{-1}(\widetilde{X}^o)}:\rho^{-1}(\widetilde{X}^o)\to \widetilde{X}^o$ is biholomorphic.
		\item $Z'$ is a functorial desingularization of the main component of $Z\times_{\widetilde{X}}X'$. In particular, $Z'\to Z\times_{\widetilde{X}}X'$ is biholomorphic over $Z^o\times_{\widetilde{X}^o}\rho^{-1}(\widetilde{X}^o)$.
	\end{itemize}
	Denote $X'^o:=\rho^{-1}(\widetilde{X}^o)$, $Z'^o:=h^{-1}(X'^o)$ and $h^o:=h|_{Z'^o}:Z'^o\to X'^o$.
	$h^o$ is a proper K\"ahler submersion which is the pullback of the family $g^o:Z^o\to\widetilde{X}^o$ via the isomorphism $\rho^o:X'^o\to\widetilde{X}^o$.
	Notice that the relative dimension of $h$ is $n$. 
	Then $R^nh^o_\ast(\bR_{Z'^o})$ underlies an $\bR$-polarized variation of Hodge structure of weight $n$ on $X'^o$. Let $(H^n_{h^o}=\bigoplus_{p=0}^n H^{p,n-p}_{h^o},\theta,h_Q)$ be the associated system of Hodge bundles with the Hodge metric $h_Q$. Namely, $H^{p,q}_{h^o}:= R^qh^o_\ast\Omega^p_{Z'^o/X'^o}$ and
	$\theta:H^{p,q}_{h^o}\to H^{p-1,q+1}_{h^o}\otimes\Omega_{X'^o}$ is defined by taking wedge product with the Kodaira-Spencer class. 
	
	Let $\omega_{X'}\simeq\rho^{\ast}\omega_{\widetilde{X}}\otimes\sO_{X'}(E')$
	for some exceptional divisor $E'$ of $\rho$.
	By Theorem \ref{thm_parabolic} and \S \ref{section_prolongation_general},
	$$\left(\psi_{\ast}\left({_{<\psi^{-1}(D_f)_{\rm red}+E'}}H^n_{h^o}\right)=\bigoplus_{p=0}^n \psi_{\ast}\left({_{<\psi^{-1}(D_f)_{\rm red}+E'}}H^{p,n-p}_{h^o}\right),\theta\right)$$is a meromorphic Higgs sheaf on $X$ such that the Higgs field $\theta$ is holomorphic over $X\backslash\pi(D_g)$ and is regular along $\pi(D_g)$.
	
	The main result of this subsection is the following theorem, whose construction is inspired by Viehweg-Zuo \cite{VZ2001}.
	\begin{thm}\label{thm_VZ_prolongation}
		Notations and assumptions as in \S \ref{section_setting} and \ref{section_VZ_sheaf}. Then the following hold.
		\begin{enumerate}
			\item There is a natural inclusion $\pi_\ast(\widetilde{L})\subset \psi_{\ast}\left({_{<\psi^{-1}(D_f)_{\rm red}+E'}}H^{n,0}_{h^o}\right)$;
			\item Let $(\bigoplus_{p=0}^n L^p,\theta)\subset \left(\psi_{\ast}\left({_{<\psi^{-1}(D_f)_{\rm red}+E'}}H^n_{h^o}\right),\theta\right)$ be the meromorphic Higgs subsheaf generated by $L^0:=\pi_\ast(\widetilde{L})$, where $$L^p\subset\psi_{\ast}\left({_{<\psi^{-1}(D_f)_{\rm red}+E'}}H^{n-p,p}_{h^o}\right).$$ Then for each $0\leq p<n$ the Higgs field $$\theta:L^p|_{X\backslash \pi(D_g)}\to L^{p+1}|_{X\backslash \pi(D_g)}\otimes \Omega_{X\backslash \pi(D_g)}$$ is holomorphic over $X\backslash D_f$ and has at most log poles along $D_f$, i.e. 
			$$\theta(L^p)\subset L^{p+1}\otimes\Omega_X(\log D_f).$$
		\end{enumerate}
	\end{thm}
	The proof of this theorem will occupy the remainder of this subsection. It will be accomplished by constructing a log Higgs subsheaf $\bigoplus_{p+q=n}G^{p,q}$ of $\psi_{\ast}\left({_{<\psi^{-1}(D_f)_{\rm red}+E'}}H^n_{h^o}\right)$ containing $\pi_\ast(\widetilde{L})$ such that the Higgs field is holomorphic on $X\backslash D_f$. We first construct the Higgs subsheaf on $\psi(X'^o)$ (over which the families we concern are smooth)  and then extend it to the whole manifold $X$ via analytic prolongations.
	\subsubsection{The construction on $\psi(X'^o)$}
	Since $\widetilde{Y}_k$ is embedded into the total space of the line bundle $\widetilde{B}$, the pullback $(\varpi\mu)^\ast \widetilde{B}$ has a tautological section. This gives an injective morphism
	\begin{align*}
	(\varpi\mu\sigma')^\ast (\widetilde{B}^{-1})\to\sO_{Z'}.
	\end{align*}
	Combining it with (\ref{align_iso_Lo_L}) one gets an injective map 
	$$\varphi^\ast(B^o)^{-1}|_{Z'^o}\simeq(\varpi\mu\sigma')^\ast (\widetilde{B}^{-1})|_{Z'^o}\to\sO_{Z'^o}.$$
	By composing it with the natural map $\varphi^\ast\Omega^p_{Y/X}\to\Omega^p_{Z'/X'}$ we obtain a natural morphism
	\begin{align}\label{align_VZ1}
	\varphi^{\ast}((B^o)^{-1}\otimes\Omega^p_{Y^o/X^o})|_{Z'^o}\to \Omega^p_{Z'^o/X'^o}
	\end{align}
	for every $p=0,\dots,n$. Denote $X_1:=\psi(X'^o)\subset X^o$ and  $Y_1:=f^{-1}(X_1)\subset Y^o$. Then $f_1:=f|_{Y_1}:Y_1\to X_1$ is a proper K\"ahler submersion.
	(\ref{align_VZ1}) induces a map
	\begin{align}\label{align_VZmap}
	\iota_{X_1}:R^qf^o_{\ast}((B^o)^{-1}\otimes\Omega^{p}_{Y^o/X^o})|_{X_1}\to \psi^o_\ast R^qh^o_\ast(\Omega^{p}_{Z'^o/X'^o})
	\end{align} 
	for every $p,q\geq0$, where $\psi^o:=\psi|_{X'^o}:X'^o\to X_1$ is an isomorphism.   
	Consider the diagram
	\begin{align*}
	\xymatrix{
		0\ar[r]&	h^{o\ast}\Omega_{X'^o}\otimes\Omega^{p-1}_{Z'^o/X'^o}\ar[r]&\Omega^{p}_{Z'^o} \ar[r] &\Omega^{p}_{Z'^o/X'^o}\ar[r]&0
	}.
	\end{align*}
	By taking the derived pushforward $R^\ast h^o_\ast$ we obtain the Higgs field as the boundary map
	\begin{align}\label{align_theta_log}
	\theta:R^qh^o_\ast(\Omega^{p}_{Z'^o/X'^o})\to R^{q+1}h^o_\ast(\Omega^{p-1}_{Z'^o/X'^o})\otimes\Omega_{X'^o}.
	\end{align}
	Consider the diagram
	\begin{align*}
	\xymatrix{
		0\ar[r]&	f^{o\ast}\Omega_{X^o}\otimes\Omega^{p-1}_{Y^o/X^o}\ar[r]&\Omega^{p}_{Y^o} \ar[r] &\Omega^{p}_{Y^o/X^o}\ar[r]&0
	}.
	\end{align*}
	By tensoring it with $(B^o)^{-1}$ and taking the derived pushforward $R^\ast f^o_{\ast}$ one has the boundary map
	\begin{align}\label{align_vartheta_nolog}
	\vartheta:R^qf^o_{\ast}((B^o)^{-1}\otimes\Omega^{p}_{Y^o/X^o})\to R^{q+1}f^o_{\ast}((B^o)^{-1}\otimes\Omega^{p-1}_{Y^o/X^o})\otimes\Omega_{X^o}.
	\end{align}
	By (\ref{align_VZ1}) there is a morphism between distinguished triangles in $D(Y_1)$
	\begin{align*}
	\xymatrix{
		f^{o\ast}\Omega_{X^o}\otimes\Omega^{p-1}_{Y^o/X^o}\otimes (B^o)^{-1}|_{Y_1}\ar[r]\ar[d]&\Omega^{p}_{Y^o}\otimes (B^o)^{-1}|_{Y_1} \ar[r]\ar[d] &\Omega^{p}_{Y^o/X^o}\otimes (B^o)^{-1}|_{Y_1}\ar[r]\ar[d]&\\
		R\varphi_\ast\left(h^{o\ast}\Omega_{X'^o}\otimes\Omega^{p-1}_{Z'^o/X'^o}\right)\ar[r]&R\varphi_\ast\left(\Omega^{p}_{Z'^o}\right) \ar[r] &R\varphi_\ast\left(\Omega^{p}_{Z'^o/X'^o}\right)\ar[r]&
	}.
	\end{align*}
	Then there is a commutative diagram
	\begin{align}\label{align_theta_vartheta_commute}
	\xymatrix{
		R^qf^o_{\ast}((B^o)^{-1}\otimes\Omega^{p}_{Y^o/X^o})|_{X_1}\ar[r]^-{\vartheta|_{X_1}}\ar[d]^{\iota_{X_1}}& R^{q+1}f^o_{\ast}((B^o)^{-1}\otimes\Omega^{p-1}_{Y^o/X^o})|_{X_1}\otimes\Omega_{X_1}\ar[d]^{\iota_{X_1}\otimes{\rm Id}}\\
		\psi^o_\ast R^qh^o_\ast(\Omega^{p}_{Z'^o/X'^o})\ar[r]^-{\theta}& \psi^o_\ast R^{q+1}h^o_\ast(\Omega^{p-1}_{Z'^o/X'^o})\otimes\Omega_{X_1}
	}.
	\end{align}
	\subsubsection{Extend the Higgs sheaves to $X^o$}
	Notice that $H^{p,q}_{h^o}\simeq R^qh^o_\ast(\Omega^p_{Z'^o/X'^o})$. The main result of this part is the following.
	\begin{lem}\label{lem_VZ_Xo}
		The map (\ref{align_VZmap}) extends to a map
		\begin{align}\label{align_VZmap2}
		\iota_{X^o}:R^qf^o_{\ast}((B^o)^{-1}\otimes\Omega^{p}_{Y^o/X^o})\to \psi_\ast\left( _{<\psi^{-1}(D_f)_{\rm red}}H^{p,q}_{h^o}\right)|_{X^o}.
		\end{align}
	\end{lem}
	\begin{proof}
		Consider the diagram
		\begin{align}\label{align_ssr_lem3_3}
		\xymatrix{
			Z''\ar[r]^-{\beta}&\varphi^{-1}(Y^o)\ar[r]^-\varphi \ar[d]^h & Y^o\ar[d]^{f^o}\\
			&\psi^{-1}(X^o)\ar[r]^-\psi & X^o
		}.
		\end{align}
		Here $h$ is a proper submersion outside the simple normal crossing divisor $\rho^{-1}(D_g)$ and
		$\beta:Z''\to \varphi^{-1}(Y^o)$ is a functorial desingularization of the pair $(\varphi^{-1}(Y^o),h^{-1}(\rho^{-1}(D_g)))$. There is a closed analytic subset $S\subset\psi^{-1}(X^o)\backslash X'^o$ so that ${\rm codim}_{\psi^{-1}(X^o)}(S)\geq 2$ and $h\beta:Z''\to \psi^{-1}(X^o)$ is semistable over $\psi^{-1}(X^o)\backslash S$.
		(\ref{align_ssr_lem3_3}) induces natural morphisms
		\begin{align}\label{align_lem33_1}
		\psi^\ast\left(R^qf^o_\ast((B^o)^{-1}\otimes\Omega^p_{Y^o/X^o})\right)\to R^qh_\ast(\Omega^p_{\varphi^{-1}(Y^o)/\psi^{-1}(X^o)})\to R^q(h\beta)_\ast\left(\Omega^p_{Z''/\psi^{-1}(X^o)}\right)
		\end{align}
		for every $p,q\geq0$. Denote 
		$X'_2:=\psi^{-1}(X^o)\backslash S$, $D_{X'_2}:=\rho^{-1}(D_g)\cap X'_2$, $Z''_2:=(h\beta)^{-1}(X'_2)$ and $D_{Z''_2}:=(h\beta)^{-1}(D_{X_2})$. Then $h\beta:(Z''_2,D_{Z''_2})\to(X'_2,D_{X'_2})$ is a proper K\"ahler semistable morphism.
		Consider the associated logarithmic Gauss-Manin connection 
		\begin{align*}
		\nabla_{\rm GM}:R^m(h\beta)_\ast\left(\Omega^\bullet_{Z''_2/X'_2}(\log D_{Z''_2})\right)\to R^m(h\beta)_\ast\left(\Omega^\bullet_{Z''_2/X'_2}(\log D_{Z''_2})\right)\otimes\Omega_{X'_2}(\log D_{X'_2}),\quad\forall m\geq0.
		\end{align*}
		According to \cite[Proposition 2.2]{Steenbrink1975}, the real parts of the eigenvalues of the residue $\nabla_{\rm GM}$ along each components of $D_{X'_2}$ lie in $[0,1)$. As a consequence, the corresponding logarithmic Higgs bundle lies in the prolongation $_{<\bm{0}}H^m$
		where
		$$H^m:=\bigoplus_{p+q=m}R^q(h\beta)_\ast\left(\Omega^p_{(Z''_2\backslash D_{Z''_2})/X'^o}\right)$$
		is the Higgs bundle associated with the proper K\"ahler submersion $Z''_2\backslash D_{Z''_2}\to X'^o$. Namely there is a natural inclusion
		\begin{align*}
		R^q(h\beta)_\ast\left(\Omega^p_{Z''_2/X'_2}(\log D_{Z''_2})\right)\to _{<\bm{0}}R^q(h\beta)_\ast\left(\Omega^p_{(Z''_2\backslash D_{Z''_2})/X'^o}\right), \quad p,q\geq0.
		\end{align*}
		Since $\beta$ is functorial, it is an isomorphism over the regular loci $h^{-1}(X'^o)$. Hence the family $Z''_2\backslash D_{Z''_2}\to X'^o$ is isomorphic to the family $h^o:Z'^o\to X'^o$. Consequently, one obtains a natural inclusion
		\begin{align*}
		R^q(h\beta)_\ast\left(\Omega^p_{Z''_2/X'_2}(\log D_{Z''_2})\right)\to _{<\psi^{-1}(D_f)_{\rm red}}H^{p,q}_{h^o}|_{X'_2}, \quad p,q\geq0.
		\end{align*}
		Taking (\ref{align_lem33_1}) into account one gets
		\begin{align}\label{align_pq_to_prolongation}
		\psi^\ast\left(R^qf^o_\ast((B^o)^{-1}\otimes\Omega^p_{Y^o/X^o})\right)\big|_{X'_2}\to _{<\psi^{-1}(D_f)_{\rm red}}H^{p,q}_{h^o}|_{X'_2}, \quad p,q\geq0.
		\end{align}
		Since $_{<\psi^{-1}(D_f)_{\rm red}}H^{p,q}_{h^o}$ is locally free (Theorem \ref{thm_parabolic}) and ${\rm codim}_{\psi^{-1}(X^o)}(S)\geq 2$, the morphism (\ref{align_pq_to_prolongation}) extends to a morphism
		\begin{align*}
		\psi^\ast\left(R^qf^o_\ast((B^o)^{-1}\otimes\Omega^p_{Y^o/X^o})\right)\to _{<\psi^{-1}(D_f)_{\rm red}}H^{p,q}_{h^o}|_{\psi^{-1}(X^o)}, \quad p,q\geq0.
		\end{align*}
		by Hartogs extension theorem. Taking the adjoint we obtain (\ref{align_VZmap2}).
	\end{proof}
	\subsubsection{Extend the Higgs sheaves to $X$}
	In this part we extend (\ref{align_VZmap2}) further to $X$. Let
	$$R^qf^o_{\ast}((B^o)^{-1}\otimes\Omega^{p}_{Y^o/X^o})\langle D_f\rangle:=\bigcup_{n\in\bZ}R^qf_{\ast}((\omega_{Y/X}\otimes f^{\ast}(L^{\vee\vee})^{-1})^{-1}\otimes\Omega^{p}_{Y/X})(nD_f)$$
	denote the sheaf of sections of $R^qf^o_{\ast}((B^o)^{-1}\otimes\Omega^{p}_{Y^o/X^o})$ that are meromorphic along $D_f$ and denote
	$$R^qh^o_\ast(\Omega^{p}_{Z'^o/X'^o})\langle\rho^{-1}(D_g)\rangle:=\bigcup_{n\in\bZ}R^qh_\ast(\Omega^{p}_{Z'/X'})(n\rho^{-1}(D_g))$$
	to be the sheaf of sections of $R^qh^o_\ast(\Omega^{p}_{Z'^o/X'^o})$ that are meromorphic along $\rho^{-1}(D_g)$. (\ref{align_theta_vartheta_commute}) naturally extends to the diagram
	\begin{align}\label{align_iota}
	\xymatrix{
		R^qf^o_{\ast}((B^o)^{-1}\otimes\Omega^{p}_{Y^o/X^o})\langle D_f\rangle\ar[r]^-{\vartheta}\ar[d]^{\iota}& R^{q+1}f^o_{\ast}((B^o)^{-1}\otimes\Omega^{p-1}_{Y^o/X^o})\langle D_f\rangle\otimes\Omega_{X}\langle D_f\rangle\ar[d]^{\iota\otimes{\rm inclusion}}\\
		\psi_\ast \big(R^qh^o_\ast(\Omega^{p}_{Z'^o/X'^o})\langle\rho^{-1}(D_g)\rangle\big)\ar[r]^-{\theta}& \psi_\ast \big(R^{q+1}h^o_\ast(\Omega^{p-1}_{Z'^o/X'^o})\langle\rho^{-1}(D_g)\rangle\otimes\Omega_{X'}\langle\rho^{-1}(D_g)\rangle\big)
	}.
	\end{align}
	Define 
	\begin{align*}
	G^{p,q}:={\rm Im}(\iota)\cap \psi_{\ast}\left({_{<\psi^{-1}(D_f)_{\rm red}+E'}}H^{p,q}_{h^o}\right).
	\end{align*}
	Since the sections of $G^{p,q}$ lie in $\psi_{\ast}\left({_{<\psi^{-1}(D_f)_{\rm red}+E'}}H^{p,q}_{h^o}\right)$, they have bounded degrees of poles along $X\backslash X_1$. As a consequence $G^{p,q}$ equals the intersection of $ \psi_{\ast}\left({_{<\psi^{-1}(D_f)_{\rm red}+E'}}H^{p,q}_{h^o}\right)$ with
	$${\rm Im}\left(R^qf_{\ast}((\omega_{Y/X}\otimes f^{\ast}(L^{\vee\vee})^{-1})^{-1}\otimes\Omega^{p}_{Y/X})(n_1D_f)\to \psi_\ast\big(R^qh_\ast(\Omega^{p}_{Z'/X'})(n_2\rho^{-1}(D_g))\big)\right)$$
	for some $n_1, n_2\in\bZ$. In particular, $G^{p,q}$ is a coherent sheaf on $X$ for every $p,q\geq0$.
	\begin{lem}
		\begin{align}\label{align_ZV_is_log}
		\theta(G^{p,q})\subset G^{p-1,q+1}\otimes\Omega_X(\log D_f).
		\end{align}
	\end{lem}
	\begin{proof}
		{\bf Case I:} Let $x\in X^o=X\backslash D_f$ and let $z_1,\dots,z_d$ be holomorphic local coordinates at $x$. It suffices to show that
		\begin{align*}
		\theta(\frac{\partial}{\partial z_i})(G^{p,q})\subset G^{p-1,q+1},\quad \forall i=1,\dots,d.
		\end{align*}
		Let $v\in R^qf^o_{\ast}((B^o)^{-1}\otimes\Omega^{p}_{Y^o/X^o})$ such that $\iota(v)\in \psi_{\ast}\left({_{<\psi^{-1}(D_f)_{\rm red}+E'}}H^{p,q}_{h^o}\right)$. One has $$\vartheta(\frac{\partial}{\partial z_i})(v)\in R^{q+1}f^o_{\ast}((B^o)^{-1}\otimes\Omega^{p-1}_{Y^o/X^o})$$ according to (\ref{align_vartheta_nolog}). Thus
		\begin{align*}
		\theta(\frac{\partial}{\partial z_i})(\iota(v))=\iota\left(\vartheta(\frac{\partial}{\partial z_i})(v)\right)\in{\rm Im}(\iota),\quad \forall i=1,\dots,d
		\end{align*}
		by (\ref{align_iota}).
		Notice that $${\rm Im}(\iota)|_{X^o}\subset \psi_{\ast}\left({_{<\psi^{-1}(D_f)_{\rm red}+E'}}H^{p,q}_{h^o}\right)|_{X^o}$$ because of Lemma \ref{lem_VZ_Xo}. This shows (\ref{align_ZV_is_log}) on $X\backslash D_f$.
		
		{\bf Case II:} Let $x\in D_f$. Let $z_1,\dots,z_d$ be holomorphic coordinates at $x$ so that $D_f=\{z_1\cdots z_l=0\}$. Denote
		\begin{align*}
		\xi_i=\begin{cases}
		z_i\frac{\partial}{\partial z_i}, & i=1,\dots,l\\
		\frac{\partial}{\partial z_i}, & i=l+1,\dots,d
		\end{cases}.
		\end{align*}
		It suffices to show that 
		\begin{align*}
		\theta(\xi_i)(G^{p,q})\subset G^{p-1,q+1},\quad \forall i=1,\dots,d.
		\end{align*}
		Let $v\in R^qf^o_{\ast}((B^o)^{-1}\otimes\Omega^{p}_{Y^o/X^o})\langle D_f\rangle$ such that $\iota(v)\in \psi_{\ast}\left({_{<\psi^{-1}(D_f)_{\rm red}+E'}}H^{p,q}_{h^o}\right)$. It follows from (\ref{align_iota}) that 
		\begin{align*}
		\theta(\xi_i)(\iota(v))=\iota\left(\vartheta(\xi_i)(v)\right)\in{\rm Im}(\iota),\quad \forall i=1,\dots,d.
		\end{align*}
		Notice that 
		\begin{align}
		\theta(\xi_i)(\iota(v))\in \psi_{\ast}\left({_{<\psi^{-1}(D_f)_{\rm red}+E'}}H^{p,q}_{h^o}\right),\quad \forall i=1,\dots,d.
		\end{align}
		by Theorem \ref{thm_parabolic}. This shows (\ref{align_ZV_is_log}) on $X$.
	\end{proof}
	\subsubsection{Final proof}	
	Because of (\ref{align_ZV_is_log}), it suffices to show the following to finish the proof of Theorem \ref{thm_VZ_prolongation}.
	\begin{lem}
		There is a natural inclusion $\pi_\ast(\widetilde{L})\subset G^{n,0}$.
	\end{lem}
	\begin{proof}
		Consider the natural map
		\begin{align*}
		\widetilde{\alpha}:\pi_\ast(\widetilde{L})\to \pi_\ast\widetilde{f}_\ast(\widetilde{f}^\ast \widetilde{L})\simeq \pi_\ast\widetilde{f}_\ast(\widetilde{B}^{-1}\otimes\omega_{\widetilde{Y}/\widetilde{X}})\subset \pi_\ast g_\ast(\omega_{Z/\widetilde{X}})\subset \psi_{\ast}\left({_{<\psi^{-1}(D_f)_{\rm red}+E'}}H^{n,0}_{h^o}\right)
		\end{align*}
		where the last inclusion is deduced from Proposition \ref{prop_prolongation0_vs_geo}. Now it suffices to show that 
		\begin{align}\label{align_lem33}
		{\rm Im}(\widetilde{\alpha})|_{X^o}\subset{\rm Im}(\iota)|_{X^o}={\rm Im}(\iota_{X^o}).
		\end{align}
		Consider the natural map
		\begin{align*}
		L^{o}\to f^o_{\ast}(f^{o\ast} L^{o})\simeq f^o_{\ast}((B^o)^{-1}\otimes\Omega^{n}_{Y^o/X^o}). 
		\end{align*}
		Since $\psi_{\ast}\left({_{<\psi^{-1}(D_f)_{\rm red}+E'}}H^{p,q}_{h^o}\right)$ is torsion free, the composition map 
		$$\alpha: L^o\to f^o_\ast((B^o)^{-1}\otimes\Omega^{n}_{Y^o/X^o})\stackrel{\iota_{X^o}}{\to} \psi_{\ast}\left({_{<\psi^{-1}(D_f)_{\rm red}+E'}}H^{n,0}_{h^o}\right)|_{X^o}\subset \psi_\ast \big(h^o_\ast(\Omega^{n}_{Z'^o/X'^o})\langle\rho^{-1}(D_g)\rangle\big)|_{X^o}$$
		is injective. So it induces an injective morphism $$L^o\to{\rm Im}\left(\iota_{X^o}:f_\ast((B^o)^{-1}\otimes\Omega^{n}_{Y^o/X^o})\to \psi_{\ast}\left({_{<\psi^{-1}(D_f)_{\rm red}+E'}}H^{n,0}_{h^o}\right)|_{X^o}\right).$$
		According to (\ref{align_commut_s_s0}), one obtains that  $\widetilde{\alpha}|_{X^o}=\alpha$. Hence we show (\ref{align_lem33}).
		This proves the lemma.
	\end{proof}
	\subsection{A meta Arakelov inequality}\label{section_AI_abstract}
	\begin{thm}\label{thm_abs_Arakelov}
		Notations and assumptions as in Theorem \ref{thm_VZ_prolongation}. Assume moreover that $X$ is a smooth projective variety of dimension $d$. Then the following hold.
		\begin{enumerate}
			\item Assume that $L^{\vee\vee}\otimes\sO_X(-D_f)$ is big. Then $\omega_X(D_f)$ is big.
			\item Assume that $\omega_X(D_f)$ is pseudo-effective. Then the following Arakelov type inequalities hold. 
			\begin{align}\label{align_Araineq_abs1}
				\mu_\alpha(L)\leq n\mu_\alpha(\omega_X(D_f))+\mu_\alpha(\sO_X(D_f))
			\end{align}
			for every movable class $\alpha\in N_1(X)$.
			\begin{align}\label{align_Araineq_abs2}
				c_1(L)A_1A_2\cdots A_{d-1}\leq \frac{n}{2}c_1(\omega_X(D_f))A_1A_2\cdots A_{d-1}
			\end{align}
			for any semiample effective divisors $A_1,\dots, A_{d-1}$ on $X$. 
		\end{enumerate}
	\end{thm}
	\begin{proof}
		{\bf Proof of (1):} Let $$\sL:=\bigoplus_{p=0}^n L^p\subset \psi_{\ast}\left({_{<\psi^{-1}(D_f)_{\rm red}+E'}}H^{n}_{h^o}\right):=\bigoplus_{p=0}^n\psi_{\ast}\left({_{<\psi^{-1}(D_f)_{\rm red}+E'}}H^{n-p,p}_{h^o}\right)$$ be the Higgs subsheaf generated by $L^0=\pi_\ast(\widetilde{L})$ as in Theorem \ref{thm_VZ_prolongation}. Then $$\pi_\ast(\widetilde{L})\otimes\sO_X(-D_f)\subset\psi_{\ast}\left({_{<E'}}H^{n,0}_{h^o}\right)$$ and $\bigoplus_{p=0}^n L^p\otimes\sO_X(-D_f)$ is a log Higgs subsheaf of $\psi_{\ast}\left({_{<E'}}H^{n}_{h^o}\right)$ such that
		\begin{align*}
			\theta(L^p\otimes\sO_X(-D_f))\subset L^{p+1}\otimes\sO_X(-D_f)\otimes\Omega_X(\log D_f),\quad\forall p=0,\dots,n-1.
		\end{align*}	
		Consider the diagram
		\begin{align*}
			L^0\otimes\sO_X(-D_f)\stackrel{\theta}{\to}L^1\otimes\sO_X(-D_f)\otimes\Omega_X(\log D_f)\stackrel{\theta\otimes{\rm Id}}{\to}L^2\otimes\sO_X(-D_f)\otimes\Omega^{\otimes 2}_X(\log D_f)\to\cdots.
		\end{align*}
		There is a minimal $n_0\leq n$ such that $L^0\otimes\sO_X(-D_f)$ is sent into
		$$\ker\left(L^{n_0}\otimes\sO_X(-D_f)\otimes\Omega^{\otimes n_0}_X(\log D_f)\to L^{n_0+1}\otimes\sO_X(-D_f)\otimes\Omega^{\otimes n_0+1}_X(\log D_f)\right)\subset K\otimes \Omega^{\otimes n_0}_X(\log D_f)$$
		where 
		$$K=\ker\left(\theta:\psi_{\ast}\left({_{<E'}}H^{n}_{h^o}\right)\to \psi_{\ast}\left({_{<E'}}H^{n}_{h^o}\otimes\Omega_{\widetilde{X}}(\log D_g)\right)\right).$$
		Since $n_0$ is minimal and $K$ is torsion free, we obtain an inclusion
		\begin{align}
			\pi_\ast(\widetilde{L})\otimes\sO_X(-D_f)\subset K\otimes \Omega^{\otimes n_0}_X(\log D_f).
		\end{align}
		This induces a nonzero morphism
		\begin{align}\label{align_send_L_to_Higgsker}
			\beta:\pi_\ast(\widetilde{L})\otimes\sO_X(-D_f)\otimes K^{\vee}\to \Omega^{\otimes n_0}_X(\log D_f).
		\end{align}
		Since $K\subset \psi_{\ast}\left({_{<E'}}H^{n}_{h^o}\right)$,  $K^{\vee}$ is weakly positive by Proposition \ref{prop_semipositive_kernel}. Since $\pi_\ast(\widetilde{L})$ and $L^{\vee\vee}$ are isomorphic in codimension one, $\pi_\ast(\widetilde{L})\otimes\sO_X(-D_f)$ is big by assumption.
		Hence $\Omega^{\otimes n_0}_X(\log D_f)$ contains the big sheaf ${\rm Im}(\beta)$. This forces that $n_0>0$. Hence $\omega_X(D_f)$ is big by \cite[Theorem 7.11]{CP2019}. This shows Claim (1). 
		
		{\bf Proof of (\ref{align_Araineq_abs1}):} The argument is divided into two cases.
		
		{\bf Case 1: $n_0=0$.} Since $\pi_\ast(\widetilde{L})$ is torsion free and is isomorphic to $L^{\vee\vee}$ in codimension one, we obtain that $(\pi_\ast(\widetilde{L}))^{\vee}\simeq L^\vee$. In this case $\theta(\pi_\ast(\widetilde{L})\otimes\sO_X(-D_f))=0$. Hence $L^\vee\otimes\sO_X(D_f)$ is weakly positive by Proposition \ref{prop_semipositive_kernel}. Now (\ref{align_Araineq_abs1}) holds since $\omega_X(D_f)$ is pseudo-effective.
		
		{\bf Case 2: $n_0\geq 1$.} 
		(\ref{align_send_L_to_Higgsker}) induces a map $$K^\vee\to {\rm Im}(\beta)^{\vee\vee}\otimes L^{\vee}\otimes\sO_X(D_f)$$
		which is surjective in codimension one. 
		Since $K^\vee$ is weakly positive, so is ${\rm Im}(\beta)^{\vee\vee}\otimes L^{\vee}\otimes\sO_X(D_f)$. Hence
		\begin{align}\label{align_psf1}
			c_1({\rm Im}(\beta))-{\rm rank}({\rm Im}(\beta))(c_1(L)-c_1(\sO_X(D_f)))
		\end{align}
		is pseudo-effective.
		Since $\omega_{X}(D_f)$ is pseudo-effective,  $c_1(\Omega^{\otimes n_0}_X(\log D_f)/{\rm Im}(\beta))$
		is pseudo-effective by \cite[Theorem 1.2]{CP2019}. Hence
		\begin{align}\label{align_psf2}
			n_0c_1(\omega_X(D_f))-c_1({\rm Im}(\beta))
		\end{align}
		is pseudo-effective. Notice that $\omega_X(D_f)$ is pseudo-effective and ${\rm rank}({\rm Im}(\beta))\geq1$. Combining (\ref{align_psf1}) with (\ref{align_psf2}) we know that
		$$nc_1(\omega_X(D_f))-c_1(L)+c_1(\sO_X(D_f))$$ is pseudo-effective. This, together with the characterization of the dual of the pseudo-effective cone \cite{BDPP2013}, proves the Arakelov type inequality (\ref{align_Araineq_abs1}). 
		
		{\bf Proof of (\ref{align_Araineq_abs2}):}
		Let $m_0=\max\{p|L^p\neq0\}$. Then there are surjective morphisms
		\begin{align*}
			\pi_\ast(\widetilde{L})\otimes T_X(-\log D_f)^{\otimes p}\to L^p\subset \psi_{\ast}\left({_{<\psi^{-1}(D_f)_{\rm red}+E'}}H^{n-p,p}_{h^o}\right),\quad p=0,\dots,m_0.
		\end{align*}
		Since $L^p$ is torsion free for every  $p=0,\dots,m_0$, these maps induce maps
		\begin{align*}
			\gamma^p:L^{\vee\vee}\otimes (L^{p})^\vee\to\Omega^{\otimes p}_X(\log D_f),\quad p=0,\dots,m_0
		\end{align*}
		which are injective in codimension one.
		Since $\omega_X(D_f)$ is pseudo-effective, it follows from \cite[Theorem 1.2]{CP2019} that the first Chern class of the quotient
		$\Omega^{\otimes p}_X(\log D_f)/{\rm Im}(\gamma^p)$ is pseudo-effective for each $p=1,\dots, m_0$. Thus
		\begin{align*}
			(pc_1(\omega_X(D_f))-{\rm rank}(L^p)c_1(L^{\vee\vee})+c_1(L^p))A_1\cdots A_{d-1}\geq0,\quad \forall p=1,\dots,m_0.
		\end{align*}
		Summing up the inequalities, we see that
		\begin{align}\label{align_inequality1}
			\left(\frac{m_0(m_0+1)}{2}c_1(\omega_X(D_f))-{\rm rank}(\sL)c_1(L)\right)A_1\cdots A_{d-1}\geq -c_1(\sL)A_1\cdots A_{d-1}.
		\end{align}
		Since $\sL$ and $\psi_{\ast}\left({_{<\psi^{-1}(D_f)_{\rm red}+E'}}H^{n}_{h^o}\right)$ are torsion free, there is a dense Zariski open subset $U\subset X$, where $X\backslash U$ has codimension $\geq2$, such that $\sL$ and $\psi_{\ast}\left({_{<\psi^{-1}(D_f)_{\rm red}+E'}}H^{n}_{h^o}\right)$ are locally free on $U$ and $\psi$ is an isomorphism over $U$. By abuse of notations we identify $U$ and $\psi^{-1}(U)$. We may assume that $[A_1]\cdots[A_{d-1}]\neq0\in N_1(X)$. Choosing $A_1,\dots, A_{d-1}$ in general positions we may assume that $C:=A_1\cap\cdots\cap A_{d-1}$ is a connected smooth curve contained in $U$ and intersects transversally with $D_f$. Now $\sL|_C$ is a log Higgs subsheaf of $$\psi_{\ast}\left({_{<\psi^{-1}(D_f)_{\rm red}+E'}}H^{n}_{h^o}\right)|_C\simeq {_{<C\cap D_f}}(H_{h^o}^n|_{C\backslash D_f}).$$
		Assume that $C\cap D_f=\{x_1,\dots,x_l\}$. The parabolic Higgs bundle $$\{{_{\sum_{i=1}^la_ix_i}}(H^n_{h^o}|_{C\backslash D_f})\}_{(a_1,\dots, a_l)\in\bR^l}$$ on $C$ is semistable with trivial parabolic degree \cite[Theorem 5]{Simpson1990}. We get therefore 
		\begin{align*}
			c_1(\sL\cap{_{\bm{0}}}(H^n_{h^o}|_{C\backslash D_f}))+\sum_{i=1}^l\sum_{0\leq\alpha<1}\alpha\dim({\rm Gr}_{\alpha}\sL_{x_i})\leq0.
		\end{align*}
		Since $\sL|_C\subset {_{<C\cap D_f}}(H^n_{h^o}|_{C\backslash D_f})$, 
		one gets that
		\begin{align}\label{align_inequality2}
			c_1(\sL)A_1\cdots A_{d-1}=c_1(\sL|_C)\leq c_1(\sL\cap{_{\bm{0}}}(H^n_{h^o}|_{C\backslash D_f}))+\sum_{i=1}^l\sum_{0\leq\alpha<1}\alpha\dim({\rm Gr}_{\alpha}\sL_{x_i})\leq0.
		\end{align}
		Notice that $m_0\leq n$ and ${\rm rank}(\sL)\geq m_0+1$ (since $\sL$ is torsion free).
		Combining (\ref{align_inequality2}) with (\ref{align_inequality1}) we obtain (\ref{align_Araineq_abs2}).
	\end{proof}
	\section{Proof of the Theorem \ref{thm_main_Arakelov_inequality}}
	\subsection{Semistable reduction in codimension one}\label{section_ssred}
	\begin{defn}[Semistable morphism in codimension one]\label{defn_semistable_cod1}
		A morphism $f:Y\to X$ between complex manifolds is {\bf semistable} (resp. {\bf strictly semistable}) if there is a (not necessarily connected) smooth divisor $D_f$ on $X$ such that the following hold.
		\begin{enumerate}
			\item $f$ is a submersion over $X\backslash D_f$ and $f^{-1}(D_f)$ is a (resp. reduced) simple normal crossing divisor on $Y$. 
			\item $f$ sends submersively any
			stratum of $f^{-1}(D_f)_{\rm red}$ onto an irreducible component of $D_f$.
		\end{enumerate}
		A morphism $f:Y\to X$ between complex manifolds is {\bf semistable in codimension one} (resp. {\bf strictly semistable in codimension one}) if there is a dense Zariski open subset $U\subset X$ with ${\rm codim}_X(X\backslash U)\geq 2$, such that $f|_{f^{-1}(U)}:f^{-1}(U)\to U$ is semistable (resp. strictly semistable).
		
		Let $f:Y\to X$ be a proper morphism which is semistable in codimension one and let $D\subset X$ be a divisor so that $f$ is a submersion over $X\backslash D$. Let $D=\cup_i D_i$ be the irreducible decomposition. The {\bf ramified divisor} $R_f$ associated with $f$ is defined to be the union of $D_i$ such that the schematic preimage $f^{-1}(x)$ is non-reduced for any general point $x\in D_i$ (i.e. $f$ is not strictly semistable along the general points of $D_i$).
	\end{defn}
	For every surjective morphism $Y\to X$ between complex spaces, denote $Y^{[r]}_X$ as the main component of the $r$-fiber product $Y\times_XY\times_X\cdots\times_XY$ (i.e. the union of irreducible components that is mapped onto $X$). Denote $f^{[r]}:Y^{[r]}_X\to X$ to be the projection map. The following proposition is known to experts. We present the proof for the convenience of readers.
	\begin{prop}\label{prop_mild_pushforward}
		Let $Y\to X$ be a strictly semistable morphism and
		$\tau:Y^{(r)}\to Y^{[r]}_X$ a desingularization. Denote $f^{(r)}=f^{[r]}\tau$. Then the following hold.
		\begin{enumerate}
			\item $\tau_{\ast}(\omega_{Y^{(r)}}^{\otimes k})\simeq\omega_{Y^{[r]}_X}^{\otimes k}$ for every $k\geq 1$, where $\omega_{Y^{[r]}_X}$ is the dualizing sheaf (invertible since $Y^{[r]}_X$ is Gorenstein).
			\item $f^{(r)}_\ast(\omega_{Y^{(r)}/X}^{\otimes k})$ is a reflexive sheaf for every $k\geq 1$;
			\item $f^{(r)}_\ast(\omega_{Y^{(r)}/X}^{\otimes k})\simeq (\otimes^rf_\ast(\omega_{Y/X}^{\otimes k}))^{\vee\vee}$ for every $k\geq 1$.
		\end{enumerate}
	\end{prop}
	\begin{proof}
		A semistable morphism is weakly semistable in the sense of Abramovich-Karu \cite{Abramovich2000}. Hence $Y^{[r]}_X$ has only normal, rational and Gorenstein singularities by \cite[Proposition 6.4]{Abramovich2000}. Thus it has canonical singularities. The first claim follows. 
		
		For the second claim, it suffices to show that any section of $f^{[r]}_\ast(\omega_{Y^{[r]}_X/X}^{\otimes k})\simeq f^{(r)}_\ast(\omega_{Y^{(r)}/X}^{\otimes k})$ extends cross an arbitrary locus of codimension $\geq 2$ . Let $U\subset X$ be an open subset and $Z\subset U$ a closed analytic subset of codimension $\geq 2$. Let $$s\in \Gamma(U\backslash Z, f^{[r]}_\ast(\omega_{Y^{[r]}_X/X}^{\otimes k}))=\Gamma((f^{[r]})^{-1}(U\backslash Z), \omega_{Y^{[r]}_X/X}^{\otimes k}).$$
		Since $f$ is flat, so is $f^{[r]}$. Hence $(f^{[r]})^{-1}(Z)$ is of codimension $\geq 2$ in $(f^{[r]})^{-1}(U)$. Since $Y^{[r]}_X$ is normal and $\omega_{Y^{[r]}_X/X}^{\otimes k}$ is invertible, by Hartog's theorem for normal complex spaces there is 
		$$\widetilde{s}\in \Gamma(U, f^{[r]}_\ast(\omega_{Y^{[r]}_X/X}^{\otimes k}))=\Gamma((f^{[r]})^{-1}(U), \omega_{Y^{[r]}_X/X}^{\otimes k})$$
		which extends $s$. This proves Claim (2).
		
		Now we show the last claim. Since $f^{[r]}$ and $f$ are Gorenstein, one obtains that 
		$$\omega_{Y^{[r]}_X/X}^{\otimes k}\simeq\otimes_{i=1}^r p_i^\ast\omega_{Y/X}^{\otimes k}$$
		where $p_i:Y^{[r]}_X\to Y$ is the projection to the $i$th component. Let $U\subset X$ be the largest open subset over which $f^{[r]}_\ast(\omega_{Y^{[r]}_X/X}^{\otimes k})$ and $f_\ast(\omega_{Y^/X}^{\otimes k})$ are locally free. Since the relevant sheaves are torsion free, $X\backslash U$ has codimension $\geq 2$. By the flat base change we obtain that
		$$f^{(r)}_\ast(\omega_{Y^{(r)}/X}^{\otimes k})|_U\simeq \otimes^r f_\ast(\omega_{Y/X}^{\otimes k})|_U.$$
		Since $f^{(r)}_\ast(\omega_{Y^{(r)}/X}^{\otimes k})$ and $ (\otimes^r f_\ast(\omega_{Y/X}^{\otimes k}))^{\vee\vee}$ are reflexive, we prove Claim (3).
	\end{proof}
	By taking the desingularizations on both the total space and the base space (c.f. \cite{Wlodarczyk2009}), every surjective proper morphism between compact complex spaces can be modified to be semistable in codimension one. 
	\begin{prop}[Semistable reduction in codimension one]
		Let $f:Y\to X$ be a proper holomorphic map between compact complex spaces. Assume that there is a closed analytic subset $Z\subset X$ containing $X_{\rm sing}$ so that $f^{-1}(X\backslash Z)$ is smooth and $f:f^{-1}(X\backslash Z)\to X\backslash Z$ is a proper submersion. Then there is a diagram
		\begin{align}
			\xymatrix{
				\widetilde{Y}\ar[dr]_{\widetilde{f}}\ar[r]^-{\tau}&Y\times_{X}{\widetilde{X}}\ar[r]\ar[d]&Y\ar[d]^f\\
				&\widetilde{X}\ar[r]^{\pi}&X
			}
		\end{align}
		such that the following hold.
		\begin{enumerate}
			\item $\widetilde{X}$ is a complex manifold. $\pi$ is a projective bimeromorphic morphism which is biholomorphic over $X\backslash Z$. $\pi^{-1}(Z)$ is a simple normal crossing divisor. When $X$ is smooth and $Z$ is a simple normal crossing divisor, one can choose $\pi$ to be the identity ${\rm Id}_X$.
			\item $\tau$ is a functorial desingularization of the main component of $Y\times_X\widetilde{X}$. In particular, $\tau$ is biholomorphic over $\pi^{-1}(X\backslash Z)\times_{X\backslash Z}f^{-1}(X\backslash Z)$.
			\item $\widetilde{f}$ is semistable in codimension one.
		\end{enumerate}
	\end{prop}
	Using Kawamata's covering trick \cite[Theorem 17]{Kawamata1981} one can modify the family further to be a strictly semistable family in codimension one.
	\begin{prop}[Strictly semistable reduction in codimension one]\label{prop_ssr_cod1}
		Let $f:Y\to X$ be a proper surjective morphism from a complex manifold $Y$ to a smooth projective variety $X$. Assume that $f$ is semistable in codimension one and there is a simple normal crossing divisor $D_f$ on $X$ such that $f$ is a submersion over $X\backslash D_f$. Then there is a commutative diagram
		\begin{align}\label{align_semistablereduction_cod1}
			\xymatrix{
				\widetilde{Y}\ar[dr]_{\widetilde{f}}\ar[r]^-{\tau}&Y\times_{X}{\widetilde{X}}\ar[r]\ar[d]&Y\ar[d]^f\\
				&\widetilde{X}\ar[r]^{\sigma_{X}}&X
			}
		\end{align}
		where 
		\begin{enumerate}
			\item $\widetilde{X}$ is a smooth projective variety,  $\sigma_{X}$ is a flat finite morphism, and  $\sigma_{X}^{-1}(D_f)$ is a simple normal crossing divisor.
			\item $\tau$ is a functorial desingularization of the main component of $Y\times_X\widetilde{X}$. In particular, $\tau$ is biholomorphic over $\sigma^{-1}_X(X\backslash D_f)\times_{X\backslash D_f}f^{-1}(X\backslash D_f)$.
			\item $\widetilde{f}$ is strictly semistable in codimension one.
		\end{enumerate}
	\end{prop}
	\begin{lem}\label{lem_functorial_pushforward_pluricanonical}
		Notations as in Proposition \ref{prop_ssr_cod1}.  Let $R_f$ be the ramified divisor associated to $f$. Let $k>0$.
		Then the pullback of forms induces an injective morphism
        $$\sigma_X^\ast \left(f_\ast(\omega_{Y/X}^{\otimes k})\otimes\sO_X(-kR_f)\otimes I_Z\right)\to \widetilde{f}_\ast(\omega^{\otimes k}_{\widetilde{Y}/\widetilde{X}})$$
        for some ideal sheaf $I_Z$ whose co-support $Z$ lies in $D_f$ and ${\rm codim}_X(Z)\geq 2$.
	\end{lem}
	\begin{proof}
		Let $Z\subset X$ be a closed algebraic subset such that ${\rm codim}(Z)\geq 2$, $f$ is semistable over $X\backslash Z$ and $\widetilde{f}$ is strictly semistable over $\sigma_X^{-1}(X\backslash Z)$.
		The family $\widetilde{f}^{-1}(\sigma^{-1}_X(X\backslash D_f))\to\sigma^{-1}_X(X\backslash D_f)$ is the base change of the smooth family $f^{-1}(X\backslash D_f)\to X\backslash D_f$. One  has therefore a natural isomorphism
		\begin{align}\label{align_pullback_pushforward2}
			\sigma_X^\ast \left(f_\ast(\omega_{f^{-1}(X\backslash D_f)/X\backslash D_f}^{\otimes k})\right)\to \widetilde{f}_\ast(\omega^{\otimes k}_{\widetilde{f}^{-1}(\sigma^{-1}_X(X\backslash D_f))/\sigma^{-1}_X(X\backslash D_f)}).
		\end{align}
		This map extends to a map
		\begin{align}
			\sigma_X^\ast \left(f_\ast(\omega_{Y/X}^{\otimes k})\otimes I_{D_f}\right)\to \widetilde{f}_\ast(\omega^{\otimes k}_{\widetilde{Y}/\widetilde{X}})
		\end{align}
		for some ideal sheaf $I_{D_f}$ with $D_f$ its co-support. To prove the lemma it suffices to show that (\ref{align_pullback_pushforward2}) extends to an injective map
		\begin{align}\label{align_pullback_pushforward3}
			\sigma_X^\ast \left(f_\ast(\omega_{Y/X}^{\otimes k})\otimes\sO_X(-kR_f)\right)|_{X\backslash Z}\to \widetilde{f}_\ast(\omega^{\otimes k}_{\widetilde{Y}/\widetilde{X}})|_{X\backslash Z}.
		\end{align}
		Without loss of generality we may assume that $f$ is semistable and $\widetilde{f}$ is strictly semistable, i.e. $Z=\emptyset$.
		Denote by $E=f^{-1}(D_f)$ (resp. $\widetilde{E}=\widetilde{f}^{-1}(D_{\widetilde{f}})$) the schematic preimage where $D_f$ (resp. $D_{\widetilde{f}}$) is the degenerate (reduced) divisor associated with $f$ (resp. $\widetilde{f}$). Let $\widetilde{\tau}:\widetilde{Y}\to Y\times_X\widetilde{X}\to Y$ be the composition map in (\ref{align_semistablereduction_cod1}). Denote 
		$\Omega_{Y/X}(\log E)$ and $\Omega_{\widetilde{Y}/\widetilde{X}}(\log \widetilde{E})$ to be the relative logarithmic cotangent bundles.
		
		The pullback of logarithmic forms gives the map
		\begin{align*}
			\widetilde{\tau}^\ast\Omega_{Y/X}(\log E)\to \Omega_{\widetilde{Y}/\widetilde{X}}(\log \widetilde{E}).
		\end{align*}
		Taking the top wedges one gets that 
		\begin{align*}
			\widetilde{\tau}^\ast(\omega_{Y}(E_{\rm red})\otimes f^\ast\omega_X(D_f)^{-1})\simeq \widetilde{\tau}^\ast\Omega^n_{Y/X}(\log E)\to\Omega^n_{\widetilde{Y}/\widetilde{X}}(\log \widetilde{E})\simeq \omega_{\widetilde{Y}}(\widetilde{E}_{\rm red})\otimes\widetilde{f}^\ast\omega_{\widetilde{X}}(D_{\widetilde{f}})^{-1}.
		\end{align*}
		Since $\widetilde{E}_{\rm red}=\widetilde{E}$ one knows that 
		\begin{align*}
			\omega_{\widetilde{Y}}(\widetilde{E}_{\rm red})\otimes\widetilde{f}^\ast\omega_{\widetilde{X}}(D_{\widetilde{f}})^{-1}\simeq \omega_{\widetilde{Y}/\widetilde{X}}.
		\end{align*}
		As a consequence, one obtains an injective map
		\begin{align*}
			\widetilde{\tau}^\ast(\omega_{Y/X}(-f^\ast R_f))\subset\widetilde{\tau}^\ast(\omega_{Y/X}(E_{\rm red}-E))\simeq\widetilde{\tau}^\ast(\omega_{Y}(E_{\rm red})\otimes f^\ast\omega_X(D_f)^{-1})\to \omega_{\widetilde{Y}/\widetilde{X}}.
		\end{align*}
		This induces the desired injective map
		\begin{align*}
			\sigma_X^\ast \left(f_\ast(\omega_{Y/X}^{\otimes k})\otimes\sO_X(-kR_f)\right)\simeq\sigma_X^\ast f_\ast(\omega_{Y/X}(-f^\ast R_f)^{\otimes k})\stackrel{\gamma}{\to} \widetilde{f}_\ast\widetilde{\tau}^\ast(\omega_{Y/X}(-f^\ast R_f)^{\otimes k})\to \widetilde{f}_\ast(\omega_{\widetilde{Y}/\widetilde{X}}^{\otimes k}).
		\end{align*}
		where $\gamma$ is the isomorphism defined by the flat base change. This shows (\ref{align_pullback_pushforward3}).
	\end{proof}
	\subsection{Proof of Theorem \ref{thm_main_Arakelov_inequality}}\label{section_proof_mainthm}
	\subsubsection{Hyperbolicity}
	\begin{thm}\label{thm_VHconj_proof}
		Let $f:Y\to X$ be a proper surjective morphism from a complex manifold $Y$ to a smooth projective variety $X$ of relative dimension $n$. Let $D_f\subset X$ be an effective divisor such that $f$ is a K\"ahler submersion over $X\backslash D_f$. Then $\omega_X(D_f)$ is big under either of the following conditions.
		\begin{itemize}
			\item There is a strictly semistable reduction $\widetilde{f}:\widetilde{Y}\to\widetilde{X}$ in codimension one of $f$ such that $\det \widetilde{f}_\ast(\omega_{\widetilde{Y}/\widetilde{X}}^{\otimes k})$ is a big line bundle for some $k>0$.
			\item $f$ is a projective morphism between smooth projective varieties with connected fibers, ${\rm Var}(f)=\dim X$ and the geometric generic fiber of $f$ admits a good minimal model.
		\end{itemize}
	\end{thm}	
	\begin{proof}
		The proof is influenced by Popa-Schnell \cite{PS17} (especially the use of the ample line bundle $M$ in Step 3). Since the second condition implies the first one (Kawamata \cite{Kawamata1985}), we only concern the first condition.
		
		{\bf Step 1: Semistable reduction in codimension one.} By resolution of singularities we may assume that $D_f$ is a simple normal crossing divisor without loss of generality. Let $\pi:Y'\to Y$ be a projective bimeromorphic morphism such that $Y'$ is smooth, $\pi$ is biholomorphic over $f^{-1}(X\backslash D_f)$ and $\pi^{-1}f^{-1}(D_f)$ is a simple normal crossing divisor on $Y'$. We may replace $Y$ by $Y'$ and assume without loss of generality that $f:Y\to X$ is semistable in codimension one.
		
		{\bf Step 2: Strictly semistable reduction in codimension one.}
		Denote $X^o=X\backslash D_f$ and $Y^o:=f^{-1}(X^o)$.
		Let $Y^{[klr]}$ be the main component of the $klr$-fiber product $Y\times_X\times\cdots\times_XY$. Let $Y^{(klr)}\to Y^{[klr]}$ be a functorial desingularization. Especially, it is biholomorphic over $Y^{o[klr]}=Y^o\times_{X^o}\times\cdots\times_{X^o}Y^o$. Denote $f^{(klr)}: Y^{(klr)}\to X$ to be the induced morphism. Let
		\begin{align*}
			\xymatrix{
				\widetilde{Y}\ar[d]^{\widetilde{f}} \ar[r]^\tau & Y \ar[d]^f\\
				\widetilde{X} \ar[r]^{\sigma} & X
			}
		\end{align*}
		be a strictly semistable reduction of $f$ in codimension one (Proposition \ref{prop_ssr_cod1}). Let $Z\subset D_f$ be an algebraic closed subset with ${\rm codim}_X(Z)\geq 2$, such that $f$ is semistable over $X\backslash Z$ and $\widetilde{f}$ is strictly semistable over $\widetilde{X}\backslash \sigma^{-1}(Z)$. Denote $\widetilde{Z}:=\sigma^{-1}(Z)$.
		
		Let $\widetilde{Y}^{[klr]}$ be the main component of the $klr$-fiber product $\widetilde{Y}\times_{\widetilde{X}}\times\cdots\times_{\widetilde{X}}\widetilde{Y}$ and $\widetilde{Y}^{(klr)}\to \widetilde{Y}^{[klr]}$ a functorial desingularization. Let  $\widetilde{f}^{(klr)}:\widetilde{Y}^{(klr)}\to \widetilde{X}$ denote the induced morphism.
		By blowing up on $\widetilde{Y}^{(klr)}$ if necessary we may assume that there is a commutative diagram
		\begin{align*}
			\xymatrix{
				\widetilde{Y}^{(klr)}\ar[r]^{\tau^{(klr)}}\ar[d]^{\tilde{f}^{(klr)}} & Y^{(klr)}\ar[d]^{f^{(klr)}}\\
				\widetilde{X} \ar[r]^\sigma& X
			}
		\end{align*}
		where $\tau^{(klr)}$ is a generically finite projective map. 
		
		{\bf Step 3:} Take an ample line bundle $M$ on $X$ so that $(\sigma_{\ast}\sO_{\widetilde{X}})^\vee\otimes M$ is globally generated. Let $A\in{\rm Pic}(X)$ so that $A\otimes\sO_X(-D_f)$ is ample. By assumption we may assume that $L_k:=\det \widetilde{f}_\ast(\omega_{\widetilde{Y}/\widetilde{X}}^{\otimes k})$ is big for some $k\geq 1$. Then there is an inclusion $\sigma^\ast(A^{\otimes k}\otimes M)\subset L_k^{\otimes kr}$ for some $r>0$. Since $f$ is semistable over $X\backslash Z$ and $\widetilde{f}$ is strictly semistable over $\widetilde{X}\backslash\widetilde{Z}$, there is an inclusion
		$$L_k^{\otimes kr}\otimes I_{\widetilde{Z}}\to I_{\widetilde{Z}}\left(\widetilde{f}_\ast(\omega_{\widetilde{Y}/\widetilde{X}}^{\otimes k})^{\otimes klr}\right)^{\vee\vee} \subset \widetilde{f}^{(klr)}_{\ast}(\omega^{\otimes k}_{\widetilde{Y}^{(klr)}/\widetilde{X}})\quad (\textrm{Proposition \ref{prop_mild_pushforward}})$$
		where $l:={\rm rank} \widetilde{f}_\ast(\omega_{\widetilde{Y}/\widetilde{X}}^{\otimes k})$ and $ I_{\widetilde{Z}}\subset\sO_{\widetilde{Z}}$ is some coherent ideal sheaf with $\widetilde{Z}$ its co-support. By \cite[Lemma 3.2]{Viehweg1983} (see also \cite[Lemma 3.1.20]{Fujino2020}), there is an inclusion
		\begin{align}\label{align_pullback_pushforward}
			\widetilde{f}^{(klr)}_\ast(\omega^{\otimes k}_{\widetilde{Y}^{(klr)}/\widetilde{X}})\subset\sigma^\ast f_\ast^{(klr)}(\omega_{Y^{(klr)}/X}^{\otimes k}).
		\end{align}
		Let $I_Z$ be an ideal sheaf with $Z$ its co-support, such that the map $\sigma^\ast(I_Z)\to\sO_{\widetilde{X}}$ factors through $I_{\widetilde{Z}}$.
		Taking the composition of the maps above we get a morphism
		$$\sigma^\ast(A^{\otimes k}\otimes M\otimes I_{Z}^{\otimes k})\to\sigma^\ast f_\ast^{(klr)}(\omega_{Y^{(klr)}/X}^{\otimes k}).$$
		This induces a map
		\begin{align*}
			A^{\otimes k}\otimes I_{Z}^{\otimes k}\to f_\ast^{(klr)}(\omega_{Y^{(klr)}/X}^{\otimes k})\otimes\sigma_\ast\sO_{\widetilde{X}}\otimes M^{-1}\subset \bigoplus f_\ast^{(klr)}(\omega_{Y^{(klr)}/X}^{\otimes k})
		\end{align*}
		since $(\sigma_{\ast}\sO_{\widetilde{X}})^\vee\otimes M$ is globally generated.
		Therefore we obtain a non-zero map
		\begin{align}\label{align_send_A_in1}
			A^{\otimes k}\otimes I_{Z}^{\otimes k}\to f_\ast^{(klr)}(\omega_{Y^{(klr)}/X}^{\otimes k}).
		\end{align}
		By the assumption on $A$ we see that
		$$(A\otimes I_{Z})^{\vee\vee}\otimes\sO_X(-D_f)\simeq A\otimes \sO_X(-D_f)$$ is ample. 
		Applying Theorem \ref{thm_abs_Arakelov}-(1) to the morphism $Y^{(klr)}\to X$ and the torsion free sheaf $A\otimes I_Z$ we obtain the theorem.
	\end{proof}
	\subsubsection{Arakelov inequalities}\label{section_Arakelov_inequality}
	Let $f:Y\to X$ be a proper surjective morphism between complex manifolds. Let $D_f\subset X$ be an effective divisor (does not necessarily have simple normal crossings) such that $f$ is a submersion over $X\backslash D_f$. Let $\pi:Y'\to Y$ be a functorial desingularization of $(Y,f^{-1}(D_f))$. In particular, $\pi$ is biholomorphic over $f^{-1}(X\backslash D_f)$ and $\pi^{-1}f^{-1}(D_f)$ is a simple normal crossing divisor. Then $Y'\to X$ is semistable in codimension one. The {\bf ramified divisor} $R_f$ associated with $\pi$ is defined to be the union of components $E$ of $D_f$ whose general fibers $\pi^{-1}f^{-1}\{x\}$ ($x\in E$) are non-reduced (i.e. $f\pi$ is not strictly semistable along the general points of $E$). 
	$R_f\subset D_f$ is a reduced divisor on $X$.
	 When $f$ is a projective morphism between quasi-projective varieties, it has been shown in \cite{Abramovich2002} that any two desingularizations of $(Y,f^{-1}(D_f))$ can be connected by a sequence of smooth blowups with centers lying over $f^{-1}(D_f)$.  Hence $R_f$ is independent of the choice of the desingularization of $Y$ at least when $f:X\to Y$ is a projective morphism between quasi-projective manifolds. 
	\begin{thm}\label{thm_Arakelov_ineq_proof}
		Let $f:Y\to X$ be a proper surjective morphism from a complex manifold $Y$ to a smooth projective variety $X$ of relative dimension $n$. Let $D_f\subset X$ be an effective divisor such that $f$ is a K\"ahler submersion over $X\backslash D_f$. Let $R_f$ be the ramified divisor associated with some functorial desingularization of $(Y,f^{-1}(D_f))$.
		Let $W\subset f_\ast(\omega_{Y/X}^{\otimes k})^{\otimes r}$ be a coherent subsheaf for some $k,r\geq 1$. Assume that $\omega_{X}(D_f)$ is pseudo-effective. Then the following Arakelov type inequalities hold.
		\begin{align}\label{align_Arakelov_inequality_proof1}
			\frac{c_1(W)A_1A_2\cdots A_{d-1}}{{\rm rank}(W)}\leq rk\left(\frac{n}{2}c_1(\omega_X(D_f))+c_1(\sO_X(R_f))\right)A_1A_2\cdots A_{d-1}
		\end{align}
		holds for any semiample effective divisors $A_1,\dots, A_{d-1}$ ($d=\dim X$) on $X$, and
		\begin{align}\label{align_Arakelov_inequality_proof2}
			\mu_\alpha(W)\leq rk\left(n\mu_\alpha(\omega_X(D_f))+\mu_\alpha(\sO_X(R_f))\right)+\frac{\mu_\alpha(\sO_X(D_f))}{{\rm rank}W}
		\end{align}
		holds for every movable class $\alpha\in N_1(X)$.
	\end{thm}
	\begin{proof}
		{\bf Step 1: Semistable reduction in codimension one.} Consider the diagram
		\begin{align*}
			\xymatrix{
				Y'\ar[d]^{f'} \ar[r]^\tau & Y \ar[d]^f\\
				X' \ar[r]^{\pi} & X
			}
		\end{align*}
		where $\pi$ is a functorial desingularization of the pair $(X,D_f)$ and $Y'$ is the functorial desingularization of the main component of $X'\times_XY$ such that $f'$ is semistable in codimension one. Notice that $D_{f'}:=\pi^{-1}(D_f)$ and the exceptional loci $Ex(\pi)$ are simple normal crossing divisors.
		Let $\omega_{X'}\simeq\pi^{\ast}\omega_X\otimes\sO_{X'}(E)$
		for some exceptional divisor $E$ of $\pi$. We obtain natural morphisms
		\begin{align*}
			\pi^\ast(f_\ast(\omega^{\otimes k}_{Y/X}))\simeq\pi^\ast(f_\ast(\omega^{\otimes k}_Y)\otimes\omega_X^{-k})\to f'_\ast(\omega^{\otimes k}_{Y'})\otimes\omega^{-k}_{X'}\otimes\sO_{X'}(kE)\simeq f'_\ast(\omega^{\otimes k}_{Y'/X'})\otimes\sO_{X'}(kE).
		\end{align*}
		It induces a natural map
		\begin{align*}
			\pi^\ast(W)\otimes\sO_{X'}(-kE)\to f'_\ast(\omega^{\otimes k}_{Y'/X'})
		\end{align*}
		which is injective over $X'\backslash E$. Denote $W'$ to be  its image. Then $$c_1(W)=\pi_\ast(c_1(W')),\quad \pi_\ast(c_1(\sO_{X'}(R_{f'})))=c_1(\sO_X(R_f)),$$ and $$c_1(\omega_X(D_f))=\pi_\ast(c_1(\omega_{X'}(D_{f'})))$$ because ${\rm codim}_X(\pi(E))\geq2$. Thus (\ref{align_Arakelov_inequality_proof1}) is equivalent to 
		\begin{align*}
			\frac{c_1(W')\pi^\ast(A_1)\cdots \pi^\ast(A_{d-1})}{{\rm rank}(W')}\leq rk\left(\frac{n}{2}c_1(\omega_{X'}(D_{f'}))+c_1(\sO_X(R_{f'}))\right)\pi^\ast(A_1)\cdots \pi^\ast(A_{d-1})
		\end{align*}
		and (\ref{align_Arakelov_inequality_proof2}) is equivalent to 
		\begin{align*}
			\mu_{\pi^\ast(\alpha)}(W')\leq rk\left(n\mu_{\pi^\ast(\alpha)}(\omega_{X'}(D_{f'}))+\mu_{\pi^\ast(\alpha)}(\sO_X(R_{f'}))\right)+\frac{\mu_{\pi^\ast(\alpha)}(\sO_{X'}(D_{f'}))}{{\rm rank}W'}.
		\end{align*}
		Hence we may replace $f:Y\to X$ by $f':Y'\to X'$, $W$ by $W'$, and $A_1,\cdots, A_{d-1},\alpha$ by their pullbacks via $\pi$. For this reason we assume that $D_f$ is a simple normal crossing divisor and $f:Y\to X$ is semistable in codimension one in the remainder of the proof.
		
		{\bf Step 2: Strictly semistable reduction in codimension one.}
		Denote $X^o=X\backslash D_f$, $Y^o:=f^{-1}(X^o)$ and $f^o:=f|_{Y^o}$. Denote $l:={\rm rank}(W)$.
		Let $Y^{[klr]}$ be the main component of the $klr$-fiber product $Y\times_X\times\cdots\times_XY$ and let $Y^{(klr)}\to Y^{[klr]}$ be a functorial desingularization which is biholomorphic over $Y^{o[klr]}=Y^o\times_{X^o}\times\cdots\times_{X^o}Y^o$. Let  $f^{(klr)}: Y^{(klr)}\to X$ denote the induced morphism. Let
		\begin{align*}
			\xymatrix{
				\widetilde{Y}\ar[d]^{\widetilde{f}} \ar[r]^\tau & Y \ar[d]^f\\
				\widetilde{X} \ar[r]^{\sigma} & X
			}
		\end{align*}
		be a strictly semistable reduction of $f$ in codimension one (Proposition \ref{prop_ssr_cod1}). Let $Z\subset X$ be an algebraic closed subset of codimension $\geq 2$, such that $f$ is semistable over $X\backslash Z$ and $\widetilde{f}$ is strictly semistable over $\widetilde{X}\backslash \sigma^{-1}(Z)$. Denote $\widetilde{Z}:=\sigma^{-1}(Z)$.
		
		Let $\widetilde{Y}^{[klr]}$ be the the main component of the $klr$-fiber product $\widetilde{Y}\times_{\widetilde{X}}\times\cdots\times_{\widetilde{X}}\widetilde{Y}$ and $\widetilde{Y}^{(klr)}\to \widetilde{Y}^{[klr]}$ a functorial desingularization. Let $\widetilde{f}^{(klr)}:\widetilde{Y}^{(klr)}\to \widetilde{X}$  denote the induced morphism.
		By blowing up on $\widetilde{Y}^{(klr)}$ (if necessary) we may assume that there is a commutative diagram
		\begin{align*}
			\xymatrix{
				\widetilde{Y}^{(klr)}\ar[r]^{\tau^{(klr)}}\ar[d]^{\tilde{f}^{(klr)}} & Y^{(klr)}\ar[d]^{f^{(klr)}}\\
				\widetilde{X} \ar[r]^\sigma& X
			}
		\end{align*}
		where $\tau^{(klr)}$ is a generically finite projective map. 
		
		{\bf Step 3: Proof of the inequalities.} 
		By Lemma \ref{lem_functorial_pushforward_pluricanonical},
		the inclusion $W\subset f_\ast(\omega_{Y/X}^{\otimes k})^{\otimes r}$ induces morphisms
		\begin{align*}
			(\det W\otimes\sO_X(-klrR_f))^{\otimes k}\otimes I_Z\to \left( f_{\ast}(\omega^{\otimes k}_{Y/X})^{\otimes klr}\otimes\sO_X(-k^2lrR_f)\otimes I_{Z'}\right)^{\vee\vee}\to f^{(klr)}_{\ast}(\omega^{\otimes k}_{Y^{(klr)}/X})
		\end{align*}
	    for some ideal sheaf $I_{\widetilde{Z}}$ whose co-support $\widetilde{Z}$ lies in $D_f$ and ${\rm codim}_X(Z)\geq2$.
		The composition of the maps above is nonzero over $X^o$ because its restriction on $X^o$ is equal to the composition of the injective maps
		$$(\det W|_{X^o})^{\otimes k}\to f^o_\ast(\omega^{\otimes k}_{Y^o/X^o})^{\otimes klr}\simeq f^{(klr)}_{\ast}(\omega^{\otimes k}_{Y^{o[klr]}/X^o}).$$
	 Applying Theorem \ref{thm_abs_Arakelov} to the morphism $Y^{(klr)}\to X$ and the torsion free sheaf $\det W\otimes\sO_X(-klrR_f)\otimes I_Z$ one obtains the theorem.
	\end{proof}
	\section{Example: semistable family of elliptic curves}\label{section_exp_elliptic_family}
	
	This section investigates the effect of the Arakelov inequality (\ref{align_main_Arakelov_ineq_H}) on the geometry of semistable families of elliptic curves, drawing inspiration from the work of Viehweg-Zuo \cite{VZ2004}. New phenomena emerge in the case of higher dimensional base spaces.
	
	Let $X$ be a smooth projective $d$-fold and let $f:Y\to X$ be a family of elliptic curves which is strictly semistable in codimension one. Let $D\subset X$ be a simple normal crossing divisor such that $f$ is smooth over $X\backslash D$. For every semiample effective divisor $A_1,\dots,A_{d-1}$, let $C=H_1\cap\cdots\cap H_{d-1}$ be a smooth curve which is an intersection of the general hypersurfaces $H_1\in|k_1A_1|,\dots, H_{d-1}\in|k_{d-1}A_{d-1}|$ for some $k_1,\dots,k_{d-1}\in\bZ^{>0}$. Let $f_C:X_C:=f^{-1}(C)\to C$ be the base change family. We assume that $f_C$ is non-isotrivial. Since $C$ is in a general position, $f_C$ is semistable. Let 
	$$\theta_C:f_{C\ast}(\omega_{X_C/C})\to  R^1f_{C\ast}(\sO_{X_C})\otimes \omega_C(\log C\cap D)$$
	be the logarithmic Kodaira-Spencer map. $\theta_C$ is an injective map since $f_C$ is non-isotrivial.
	Denote $N_{C/X}\simeq\oplus_{i=1}^{d-1}\sO_X(H_i)|_C$ to be the normal bundle.
	\begin{thm}
		Assume that $f_C$ is non-isotrivial and $\omega_S(D)$ is pseudo-effective. Then we have the inequality
		\begin{align}\label{align_ssfamily_ellipt_Arakelov}
			c_1(f_\ast(\omega_{Y/X}))A_1\cdots A_{d-1}\leq\frac{1}{2}c_1(\omega_X(D))A_1\cdots A_{d-1}.
		\end{align}
	    It is an equality if and only if $\deg({\rm coker}(\theta_C))=\deg N_{C/X}$.
	    Moreover we have the following.
	    \begin{itemize}
        \item Assume that $\deg N_{C/X}=0$ and (\ref{align_ssfamily_ellipt_Arakelov}) is an equality. Then $f|_C$ is modular in the sense that $C\backslash D$ is the quotient of the upper half plane $\bH$ by a subgroup of $SL_2(\bZ)$ of finite index, and the morphism
        $C\backslash D\to \bH/SL_2(\bZ)$ is given by the $j$-invariant of the fibres.
        \item If $\deg N_{C/X}\neq0$, then the peroid map $\tau:C\backslash D\to \bH/SL_2(\bZ)$ is a (possibly ramified) covering map with at most $\deg N_{C/X}$ ramified points.
        \end{itemize}
	\end{thm}
    \begin{proof}
    	Although (\ref{align_ssfamily_ellipt_Arakelov}) is a direct consequence of Theorem \ref{thm_main_Arakelov_inequality}, we give a simplified proof which explains the spirit of the proof of Theorem \ref{thm_main_Arakelov_inequality}. We may remove a loci on $X$ of codimension two without changing the first Chern classes of the relevant sheaves. Hence we assume that $D\subset X$ is a smooth divisor and $f$ is semistable.
    	Consider the logarithmic Gauss-Manin connection $$\nabla_{\rm GM}:R^1f_\ast(\Omega_{X/S}^{\bullet}(\log))\to R^1f_\ast(\Omega_{X/S}^{\bullet}(\log))\otimes\Omega_S(\log D)$$
    	and the associated logarithmic Higgs bundle (i.e. the total graded quotient with respect to the Hodge filtration)
    	$$f_\ast(\omega_{X/S})\oplus R^1f_\ast(\sO_X)\stackrel{\theta}{\to}\left(f_\ast(\omega_{X/S})\oplus R^1f_\ast(\sO_X)\right)\otimes\Omega_S(\log D)$$
    	determined by $\theta(R^1f_\ast(\sO_X))=0$ and
    	$$f_\ast(\omega_{X/S})\stackrel{\theta}{\to} R^1f_\ast(\sO_X)\otimes\Omega_S(\log D).$$
    	Denote 
    	$$\rho:f_\ast(\omega_{X/S})\otimes T_S(-\log D)\to R^1f_\ast(\sO_X).$$
    	Then $f_\ast(\omega_{X/S})\oplus{\rm Im}(\rho)$ is a Higgs subsheaf of $f_\ast(\omega_{X/S})\oplus R^1f_\ast(\sO_X)$. According to Simpson \cite{Simpson1990} and Mochizuki \cite{Mochizuki20072}, $f_\ast(\omega_{X/S})\oplus R^1f_\ast(\sO_X)$ is a semistable Higgs bundle with vanishing Chern classes. Hence
    	\begin{align}\label{align_Arakelov_ell1}
    		c_1(f_\ast(\omega_{X/S}))\cdot C+c_1({\rm Im}(\rho))\cdot C\leq 0.
    	\end{align}
        The map $\rho$ induces a morphism
        $$\gamma:f_\ast(\omega_{X/S})\oplus {\rm Im}(\rho)^{\vee}\to\Omega_S(\log D)$$
        which is injective in codimension one. Since $\omega_S(D)$ is pseudo-effective, it follows from \cite[Theorem 1.2]{CP2019} that the first Chern classes of the quotients
        $\Omega_S(\log D)/{\rm Im}(\gamma)$ are pseudo-effective. Hence  we have
        \begin{align}\label{align_Arakelov_ell2}
        	c_1(f_\ast(\omega_{X/S}))\cdot C-c_1({\rm Im}(\rho))\cdot C\leq c_1(\omega_S(\log D))\cdot C.
        \end{align}
        By (\ref{align_Arakelov_ell1}) and (\ref{align_Arakelov_ell2}) one shows the first claim of the theorem:
        \begin{align}\label{align_Arakelov_elliptic}
        	2c_1(f_\ast(\omega_{X/S}))\cdot C\leq c_1(\omega_S(\log D))\cdot C.
        \end{align}
        Next we prove the second claim. Let $\pi:H\to C\backslash D$ be the universal covering map and $\bH\subset\bC$ the upper half plane. Let $\rho:H\to\bH$ be the peroid map defined by the variation of Hodge structure of weight $1$ associated to $f|_{C\backslash D}$. Let
        $$\tau:T_{C}(-\log D\cap C)\to Hom(f_{C\ast}(\omega_{X_C/C}),R^1f_{C\ast}(\sO_{X_C}))$$
        be determined by $\theta_C$. Then $\pi^\ast(\tau)=d\rho:T_H\to T_{\bH}$ is the tangent map of the peroid map.
        
        If $\deg N_{C/X}=0$, one has $c_1(\omega_S(\log D))\cdot C=\deg\omega_C(\log C\cap D)$. According to Viehweg-Zuo \cite{VZ2004}, (\ref{align_Arakelov_elliptic}) is an equality if and only if $f_C:X_C\to C$ is modular.
        
        If $\deg N_{C/X}>0$, then the critical points of the period map $\tau:C\backslash D\to \bH/SL_2(\bZ)$ lie in the support of ${\rm coker}(\tau)$. Hence there are at most $\deg{\rm coker}(\tau)=\deg N_{C/X}$ critical points.
    \end{proof}
	\section{Boundedness of polarized algebraic families}\label{section_boundedness}
	This section is devoted to the proof of the second Arakelov type inequality (Theorem \ref{thm_main_numbound_moduli}, see also Theorem \ref{thm_numbound_polarization}).
	In conjunction with the criterion of Kov\'acs-Lieblich \cite{Kovacs2011},
	it enables us to generalize the boundedness part of the Par\v{s}in-Arakelov-Shafarevich package to an arbitrary Kodaira dimension. The main boundedness results in the present paper are Theorem \ref{thm_bounded_stable_family} and Theorem \ref{thm_bounded_stable_family_strong}.
	\subsection{Stable minimal models and their moduli}\label{section_moduli}
	We review the main results of Birkar \cite{Birkar2022} that will be used in the sequel. A {\bf stable minimal model} is a triple $(X,B),A$ where $X$ is a reduced connected projective scheme of finite type over $\bC$ and $A,B\geq0$ are $\bQ$-divisor such that the following hold.
	\begin{itemize}
		\item $(X,B)$ is a projective connected slc pair,
		\item $K_X+B$ is semi-ample,
		\item $K_X+B+tA$ is ample for some $t>0$, and
		\item $(X,B+tA)$ is slc for some $t>0$.
	\end{itemize}
	Let
	$$d\in\bN,c\in\bQ^{\geq0},\Gamma\subset\bQ^{>0} \textrm{ a finite set, and }\sigma\in\bQ[t].$$
	A $(d,\Phi_c,\Gamma,\sigma)$-stable minimal model is a stable minimal model $(X,B),A$ such that the following hold.
	\begin{itemize}
		\item $\dim X=d,$
		\item the coefficients of $A$ and $B$ are in $c\bZ^{\geq0}$.
		\item ${\rm vol}(A|_F)\in\Gamma$ where $F$ is any general fiber of the fibration $f:X\to Z$ determined by $K_X+B$.
		\item ${\rm vol}(K_X+B+tA)=\sigma(t)$ for $0\leq t\ll 1$.
	\end{itemize}
	Let $S$ be a reduced scheme over $\bC$. A family of $(d,\Phi_c,\Gamma,\sigma)$-stable minimal models over $S$ consists of a projective morphism $X\to S$ of schemes and $\bQ$-divisors $A$ and $B$ on $X$ such that the following hold.
	\begin{itemize}
		\item $(X,B+tA)\to S$ is a locally stable family (i.e. $K_{X/S}+B+tA$ is $\bQ$-Cartier) for every sufficiently small rational number $t\geq0$,
		\item $A=cN$, $B=cD$ where $N, D\geq0$ are relative Mumford divisors,
		\item $(X_s,B_s), A_s$ is a $(d,\Phi_c,\Gamma,\sigma)$-stable minimal model for each point $s\in S$.
	\end{itemize}
	Let ${\rm Sch}_{\bC}^{\rm red}$ denote the category of reduced schemes defined over $\bC$. Define the functor of groupoids over ${\rm Sch}_{\bC}^{\rm red}$:
	$$\sM^{\rm red}_{\rm slc}(d,\Phi_c,\Gamma,\sigma): S\mapsto\{\textrm{family of } (d,\Phi_c,\Gamma,\sigma)-\textrm{stable minimal models over }S\}.$$
	\begin{thm}[Birkar \cite{Birkar2022}]\label{thm_moduli_stable_var}
		There is a proper Deligne-Mumford stack $\sM_{\rm slc}(d,\Phi_c,\Gamma,\sigma)$ over $\bC$ such that the following hold.
		\begin{itemize}
			\item $\sM_{\rm slc}(d,\Phi_c,\Gamma,\sigma)|_{{\rm Sch}_{\bC}^{\rm red}}=\sM^{\rm red}_{\rm slc}(d,\Phi_c,\Gamma,\sigma)$ as functors of groupoids.
			\item $\sM_{\rm slc}(d,\Phi_c,\Gamma,\sigma)$ admits a  projective good coarse moduli space $M_{slc}(d,\Phi_c,\Gamma,\sigma)$.
		\end{itemize}
	\end{thm}
	\begin{proof}
		See the proof of \cite[Theorem 1.14]{Birkar2022}. Following the notations in \cite[\S 10.7]{Birkar2022}, we have $$\sM_{\rm slc}(d,\Phi_c,\Gamma,\sigma)=\left[M_{\rm slc}^e(d,\Phi_c,\Gamma,\sigma,a,r,\bP^n)/{\rm PGL}_{n+1}(\bC)\right]$$
		where the right hand side is the stacky quotient.
	\end{proof}
	A stable minimal model $(X,B),A$ is called a lc stable minimal model if $(X,B)$ is a lc pair.
	Let  $\sM_{\rm lc,(0,1)}(d,\Phi_c,\Gamma,\sigma)\subset \sM_{\rm slc}(d,\Phi_c,\Gamma,\sigma)$ denote the open substack consisting of  $(d,\Phi_c,\Gamma,\sigma)$-lc stable minimal models $(X,B),A$ such that the coefficients of $B$ lie in $(0,1)$. Denote $M_{\rm lc,(0,1)}(d,\Phi_c,\Gamma,\sigma)$ to be the quasi-projective coarse moduli spaces of $\sM_{\rm lc,(0,1)}(d,\Phi_c,\Gamma,\sigma)$.
	\subsection{Polarization on $M_{\rm slc}(d,\Phi_c,\Gamma,\sigma)$}\label{section_polarization_moduli}
	In this part we consider some natural ample line bundles on $M_{\rm slc}(d,\Phi_c,\Gamma,\sigma)$. Their constructions are implicit in the proof of \cite[Theorem 1.14]{Birkar2022}, based on the ampleness criterion of Koll\'ar \cite{Kollar1990}. 
	Fix a data $d,\Phi_c,\Gamma,\sigma$. Since $\sM_{\rm slc}(d,\Phi_c,\Gamma,\sigma)$ is of finite type, there are constants $$(a,r,j)\in\bQ^{\geq0}\times(\bZ^{>0})^{2}$$ depending only on $d,\Phi_c,\Gamma,\sigma$ such that the following hold for every $(d,\Phi_c,\Gamma,\sigma)$-stable minimal model $(X,B),A$ (c.f. \cite[Lemma 10.2]{Birkar2022}).
	\begin{itemize}
		\item $X+B+aA$ is slc,
		\item $r(K_X+B+aA)$ is a very ample integral Cartier divisor with 
		$$H^i(X,kr(K_X+B+aA))=0,\quad\forall i>0,\forall k>0,$$
		\item the embedding $X\hookrightarrow\bP(H^0(X,r(K_X+B+aA)))$ is defined by degree $\leq j$ equations,
		\item the multiplication map $$S^j(H^0(X,r(K_X+B+aA)))\to H^0(X,jr(K_X+B+aA))$$ is surjective.
	\end{itemize}
	\begin{defn}
		$(a,r,j)\in\bQ^{\geq0}\times(\bZ^{>0})^{2}$ that satisfies the conditions above is called a {\bf $(d,\Phi_c,\Gamma,\sigma)$-polarization data}.
	\end{defn}
	Let $(a,r,j)\in\bQ^{\geq0}\times(\bZ^{>0})^{2}$ be a $(d,\Phi_c,\Gamma,\sigma)$-polarization data.
	Let $(X,B),A\to S$ be a family of $(d,\Phi_c,\Gamma,\sigma)$-stable minimal models. Then $f_\ast(r(K_{X/S}+B+aA))$ is locally free and commutes with an arbitrary base change. Therefore the assignment
	$$f:(X,B),A\to S\in\sM_{\rm slc}(d,\Phi_c,\Gamma,\sigma)(S)\mapsto f_\ast(r(K_{X/S}+B+aA))$$
	gives a locally free coherent sheaf on the stack $\sM_{\rm slc}(d,\Phi_c,\Gamma,\sigma)$, which is denoted by $\Lambda_{a,r}$. Let $\lambda_{a,r}:=\det(\Lambda_{a,r})$. Since $\sM_{\rm slc}(d,\Phi_c,\Gamma,\sigma)$ is Deligne-Mumford, some power $\lambda_{a,r}^{\otimes k}$ descends to a line bundle on $M_{\rm slc}(d,\Phi_c,\Gamma,\sigma)$. For this reason we regard $\lambda_{a,r}$ as a $\bQ$-line bundle on $M_{\rm slc}(d,\Phi_c,\Gamma,\sigma)$. 
	\begin{prop}\label{prop_ample_line_bundle_moduli}
		Let $(a,r,j)\in\bQ^{\geq0}\times(\bZ^{>0})^{2}$ be a $(d,\Phi_c,\Gamma,\sigma)$-polarization data. Then $\lambda_{a,r}$ is ample on $M_{\rm slc}(d,\Phi_c,\Gamma,\sigma)$.
	\end{prop}
	\begin{proof}
		By the same arguments as in \cite[\S 2.9]{Kollar1990}, it suffices to show that $f_\ast(r(K_{X/S}+B+aA))$ is nef when $S$ is a smooth projective curve. This is accomplished by Fujino \cite{Fujino2018} and Kov\'acs-Patakfalvi \cite{Kovacs2017}.
	\end{proof}
	\subsection{Numerical bound of $\lambda_{a,r}$}
	Recall that a family $f:(X,\Delta)\to S$ is log smooth if $X$ is smooth projective over $S$ and $\Delta$ is a simple normal crossing $\bQ$-divisor whose strata are all smooth over $S$.
	\begin{defn}\label{defn_admissible}
		Let $f:(X,B),A\to S$ be a family of stable minimal model. A log smooth birational model of $f$ is a log smooth family $(X',\Delta')\to S$ together with a birational map $g:X'\dashrightarrow X$ over $S$ such that $g$ is defined on a dense Zariski open subset of $\Delta'$ and $g_\ast(\Delta')=A+B$. $f$ is called {\bf admissible} if it admits a log smooth birational model and the coefficients of $B$ lie in $(0,1)$.
	\end{defn}
	The following lemma will be used in the proof of Theorem \ref{thm_numbound_polarization}.
	\begin{lem}[Relative Kawamata's covering]\label{lem_rel_Kawamata_covering}
		Let $X\to S$ be a morphism between smooth projective varieties and $D$ a simple normal crossing divisor whose coefficients lie in $(0,1)$. Let $S^o\subset S$ be a Zariski open subset such that $(f^{-1}(S^o),D\cap f^{-1}(S^o))\to S^o$ is a log smooth family. Then there is a finite ramified covering $h:Y\to X$ such that the following hold.
		\begin{enumerate}
			\item $Y$ is a smooth projective variety,
			\item $f\circ h$ is a smooth morphism over $S^o$,
			\item there is a $\bQ$-divisor $F\geq0$ such that 
			\begin{align}\label{align_canonical_bundle_formula}
			h^\ast(K_X+D)=K_Y-F
			\end{align}
		\end{enumerate}
	\end{lem}
	\begin{proof}
		The proof is the same as \cite[Theorem 17]{Kawamata1981} except two modifications. The first is that the general hyperplanes $H_1,\dots,H_d$ in loc. cit.  should satisfy that $$\left(H_{1}\cup\dots\cup H_{d}\cup D_{1}\cup\dots\cup D_{l}\right)\cap f^{-1}(S^o)$$ is a relative simple normal crossing divisor over $S^o$. The second is that one should let $m_i$ be sufficiently large in loc. cit. in order to ensure the validity of (\ref{align_canonical_bundle_formula}).
	\end{proof}
	\begin{thm}[Uniform numerical bound of the polarization]\label{thm_numbound_polarization}
		Let $f^o:(X^o,B^o),A^o\to S^o$ be an admissible family of $(d,\Phi_c,\Gamma,\sigma)$-lc stable minimal models over a smooth quasi-projective variety $S^o$. 
		Let $S$ be a smooth projective variety containing $S^o$ as a Zariski open subset. Assume that $D:=S\backslash S^o$ is a divisor and the morphism $\xi^o:S^o\to M_{\rm lc,(0,1)}(d,\Phi_c,\Gamma,\sigma)$ induced from the family $f^o$ extends to a morphism $\xi:S\to M_{\rm slc}(d,\Phi_c,\Gamma,\sigma)$. Let $(a,r,j)\in\bQ^{\geq0}\times(\bZ^{>0})^{2}$ be a $(d,\Phi_c,\Gamma,\sigma)$-polarization data satisfying that  $B^o+aA^o<(A^o+B^o)_{\rm red}$. Assume that $K_{S}+D$ is pseudo-effective. Then the following inequalities hold.
		\begin{align*}
			c_1(\xi^\ast\lambda_{a,r})A_1A_2\cdots A_{\dim S-1}\leq \frac{rd{\rm rank}(\Lambda_{a,r})}{2}(K_S+D)A_1A_2\cdots A_{\dim S-1}
		\end{align*}
		for any semiample effective divisors $A_1,\dots, A_{\dim S-1}$ on $S$, and
		\begin{align*}
			c_1(\xi^\ast\lambda_{a,r})\cdot\alpha\leq rd{\rm rank}(\Lambda_{a,r})\left(K_S+D\right)\cdot\alpha+D\cdot\alpha
		\end{align*}
		for every movable class $\alpha\in N_1(S)$. If in particular $\dim S=1$, then
		\begin{align}\label{align_polarization_bound3}
			\deg(\xi^\ast\lambda_{a,r})\leq \frac{rd{\rm rank}(\Lambda_{a,r})}{2}\deg(K_S+D).
		\end{align}
	\end{thm}
	\begin{proof}
		{\bf Step 1 (Compactifying the family):}
		By the properness of $\sM_{\rm slc}(d,\Phi_c,\Gamma,\sigma)$ there are data as follows.
		\begin{itemize}
			\item A proper generically finite morphism $\sigma:\widetilde{S}\to S$ from a smooth projective variety $\widetilde{S}$. $\sigma$ is a combination of smooth blowups and a finite flat morphism such that $\sigma^{-1}(D)$ is a simple normal crossing divisor. Note that $\sigma$ might be ramified over $S^o$.
			\item Denote $\widetilde{S}^o:=\sigma^{-1}(S^o)$, $\widetilde{X}^o:=\widetilde{S}^o\times_{S^o}X^o$. Denote $\widetilde{A}^o$ and $\widetilde{B}^o$ to be the divisorial pullbacks of $A^o$ and $B^o$ respectively. There is a compactification $\widetilde{f}:(\widetilde{X},\widetilde{B}),\widetilde{A}\to\widetilde{S}$ of the base change family $(\widetilde{X}^o,\widetilde{B}^o),\widetilde{A}^o\to \widetilde{S}^o$ such that 
			$\widetilde{f}\in\sM_{\rm slc}(d,\Phi_c,\Gamma,\sigma)(\widetilde{S})$.
		\end{itemize}
		{\bf Step 2 (Taking the log smooth birational models):} 
		Consider the following commutative diagram
		\begin{align*}
		\xymatrix{
			X'^o\ar@{-->}[d]^{\rho^o}\ar@/_/[dd]_-{f'^o} &\widetilde{X}'^o \ar@/_/[dd]_-{\widetilde{f}'^o} \ar[l]\ar@{-->}[d]^{\widetilde{\rho}^o} \ar[r]^{\subset} & \widetilde{X}'\ar@{-->}[d]^{\widetilde{\rho}}\ar@/_/[dd]_-{\widetilde{f}'}\\
			X^o\ar[d]^{f^o} &\widetilde{X}^o \ar[l]\ar[d] \ar[r]_{\subset} & \widetilde{X}\ar[d]^{\widetilde{f}}\\
			S^o& \widetilde{S}^o\ar[l]\ar[r]^{\subset} & \widetilde{S}
		}
		\end{align*}
		where the arrows are explained as follows.
		\begin{itemize}
			\item $f'^o:X'^o\to S^o$ is a projective smooth morphism and $\rho^o:X'^o\to X^o$ is a rational map over $S$ whose image contains an dense Zariski open subset of $A^o+B^o$. Denote the $\bQ$-divisor $A'^o$ and $B'^o$ on $X'^o$ to be the birational transforms of $A^o$ and $B^o$ respectively. $(X'^o,A'^o+B'^o)\to S^o$ is a log smooth birational model of $f^o:(X^o,B^o),A^o\to S^o$.
			\item $\widetilde{f}'^o$ and $\widetilde{\rho}^o$ are the base changes of $f'^o$ and $\rho^o$ respectively. Denote  $\widetilde{A}'^o$ and $\widetilde{B}'^o$ to be the birational transforms of $\widetilde{A}^o$ and $\widetilde{B}^o$. Then $\widetilde{f}'^o:(\widetilde{X}'^o,\widetilde{A}'^o+\widetilde{B}'^o)\to \widetilde{S}^o$ is a log smooth birational model of $(\widetilde{X}^o,\widetilde{B}^o),\widetilde{A}^o\to \widetilde{S}^o$.
			\item $\widetilde{f}':\widetilde{X}'\to \widetilde{S}$ is a completion of $\widetilde{f}'^o$. Since $\widetilde{A}$ and $\widetilde{B}$ do not contain any component of a fiber of $\widetilde{f}$, $\widetilde{\rho}^o$ extends naturally to a rational map $\widetilde{\rho}:\widetilde{X}'\to\widetilde{X}$ whose image contains a dense Zariski open subset of $\widetilde{A}+\widetilde{B}$. Hence we can define the birational transforms of $\widetilde{A}$ and $\widetilde{B}$, denoted by $\widetilde{A}'$ and $\widetilde{B}'$ respectively. By blowing up along the 
			centers over $\widetilde{f}'^{-1}\sigma^{-1}(D)$ we may assume the following. 
			\begin{itemize}
				\item $\widetilde{X}'$ is a smooth projective variety,
				\item $\widetilde{f}'^{-1}\sigma^{-1}(D)+ \widetilde{A}'+\widetilde{B}'$ is a simple normal crossing divisor, and
				\item $\widetilde{f}':\widetilde{X}'\to \widetilde{S}$ is smooth over $\widetilde{S}^o$.
			\end{itemize}
		\end{itemize}
		The constructions yields that 
		\begin{align}\label{align_NK}
		\widetilde{f}_{\ast}(r(K_{\widetilde{X}/\widetilde{S}}+\widetilde{B}+a\widetilde{A}))\simeq \widetilde{f}'_{\ast}(r(K_{\widetilde{X}'/\widetilde{S}}+\widetilde{B}'+a\widetilde{A}'))
		\end{align}
		whenever $r(K_{\widetilde{X}/\widetilde{S}}+\widetilde{B}+a\widetilde{A})$ is integral.
		
		{\bf Step 3 (Kawamata's covering):} 
		Let $(a,r,j)\in\bQ^{\geq0}\times(\bZ^{>0})^{2}$ be a $(d,\Phi_c,\Gamma,\sigma)$-polarization data with $0<a\ll 1$ so that $\widetilde{B}'+a\widetilde{A}'$ is a simple normal crossing divisor with coefficients lie in $(0,1)$. By Lemma \ref{lem_rel_Kawamata_covering} there is a finite ramified covering $\varrho:\widetilde{Y}'\to \widetilde{X}'$ such that the following hold.
		\begin{enumerate}
			\item $\widetilde{Y}'$ is a smooth projective variety,
			\item $\widetilde{f}'\varrho$ is smooth over $S^o$,
			\item there is a non-negative $\bQ$-divisor $F$ such that 
			\begin{align}\label{align_adjunction_to_smooth}
			\varrho^\ast(K_{\widetilde{X}'}+\widetilde{B}'+a\widetilde{A}')=K_{\widetilde{Y}'}-F.
			\end{align}
		\end{enumerate}
		The same construction is valid on $X'^o$. One can choose a suitable Kawamata covering of $Y^o\to X'^o$ such that $\widetilde{f}'\varrho:\widetilde{Y}'\to \widetilde{S}$ is a compactification of the base change morphism $Y^o\times_{S^o}\widetilde{S}^o\to\widetilde{S}^o$. 
		Consider the following diagram
		\begin{align}
		\xymatrix{
			\widetilde{Y}'\ar[dr]_{\widetilde{f}'\varrho}&\widetilde{Y}\ar[r]\ar[l] & Y\ar[d]\\
			&\widetilde{S}\ar[r]^{\sigma} & S
		},
		\end{align}
		where $Y\to S$ is a completion of $Y^o\to S^o$ with $Y$ smooth and projective, $\widetilde{Y}\to\widetilde{Y}'$ is some modification which is biholomorphic over $Y^o\times_{S^o}\widetilde{S}^o$ so that there is a morphism $\widetilde{Y}\to Y$ making the diagram commutative. We may assume that $\sigma$ is sufficiently ramified so that $\widetilde{Y}\to\widetilde{S}$ is semistable in codimension one.
		Denote $g:Y\to S$ and $\widetilde{g}:\widetilde{Y}\to\widetilde{S}$ to be the maps we have just constructed. According to (\ref{align_NK}) and (\ref{align_adjunction_to_smooth}) there is an inclusion
		\begin{align}\label{align_put_polarization_to_pushforward1}
		\widetilde{f}_\ast(rK_{\widetilde{X}/\widetilde{S}}+r\widetilde{B}+ra\widetilde{A})\subset \widetilde{g}_\ast(rK_{\widetilde{Y}/\widetilde{S}}).
		\end{align}
		By the same argument we could obtain an inclusion
		\begin{align}\label{align_put_polarization_to_pushforward2}
			f^o_\ast(rK_{X^o/S^o}+rB^o+raA^o)\subset g^o_\ast(rK_{Y^o/S^o})
		\end{align}
		where $g^o:Y^o\to X'^o\to S^o$ is the composition map. The two inclusions (\ref{align_put_polarization_to_pushforward1}) and (\ref{align_put_polarization_to_pushforward2}) are compatible via the base change.
		
		{\bf Last step:} We finish the proof by applying Theorem \ref{thm_abs_Arakelov}. The proof is parallel to the proof of Theorem \ref{thm_Arakelov_ineq_proof}.
		Denote
		$$W:=\widetilde{f}_\ast(rK_{\widetilde{X}/\widetilde{S}}+r\widetilde{B}+ra\widetilde{A})\quad\textrm{and}\quad l:={\rm rank}(W).$$
		By the construction of $\lambda_{a,r}$ one has
		\begin{align*}
		(\xi\circ\sigma)^\ast\lambda_{a,r}\simeq\det(W).
		\end{align*}
		Since $\xi^\ast(\lambda_{a,r})$ is a $\bQ$-line bundle, there is $k\in\bZ^{>0}$ such that $\xi^\ast(\lambda_{a,r})^{\otimes k}$ is a line bundle.
		Let $\widetilde{Y}^{(klr)}$ denote a functorial desingularization of the main component of the $klr$-fiber product $\widetilde{Y}\times_{\widetilde{S}}\times\cdots\times_{\widetilde{S}}\widetilde{Y}$ and denote  $\widetilde{g}^{(klr)}:\widetilde{Y}^{(klr)}\to \widetilde{S}$ to be the projection. Define $g^{(klr)}:Y^{(klr)}\to S$ similarly. We may assume that there is a morphism $\widetilde{Y}^{(klr)}\to Y^{(klr)}$ making the diagram
		\begin{align}
			\xymatrix{
				\widetilde{Y}^{(klr)}\ar[r]\ar[d]^{\widetilde{g}^{(klr)}} & Y^{(klr)} \ar[d]^{g^{(klr)}}\\
				\widetilde{S}\ar[r]^{\sigma}&S
			}
		\end{align}
		commutative.
		According to (\ref{align_put_polarization_to_pushforward1}) there is an inclusion
		$W\subset \widetilde{g}_\ast(rK_{\widetilde{Y}/\widetilde{S}})$.
		Since $\widetilde{Y}\to \widetilde{S}$ is smooth over $\widetilde{S}^o$ and $\widetilde{Y}\to\widetilde{S}$ is semistable in codimension one, by Proposition \ref{prop_mild_pushforward} we get a natural inclusion
		\begin{align*}
		\det(W)^{\otimes kr}\otimes I_{\widetilde{Z}}\to \left(\widetilde{g}_{\ast}(rK_{\widetilde{Y}/\widetilde{S}})^{\otimes klr}\right)^{\vee\vee}\otimes I_{\widetilde{Z}}\subset\widetilde{g}^{(klr)}_{\ast}(rK_{\widetilde{Y}^{(klr)}/\widetilde{S}}),
		\end{align*}
		where $I_{\widetilde{Z}}$ is an ideal sheaf whose co-support $\widetilde{Z}$ lies in $\sigma^{-1}(D)$ and ${\rm codim}_{\widetilde{S}}(\widetilde{Z})\geq 2$. 
		By \cite[Lemma 3.2]{Viehweg1983} (see also \cite[Lemma 3.1.20]{Fujino2020}), there is an inclusion
		\begin{align*}
		\widetilde{g}^{(klr)}_\ast(rK_{\widetilde{Y}^{(klr)}/\widetilde{S}})\subset\sigma^\ast g_\ast^{(klr)}(rK_{Y^{(klr)}/S}).
		\end{align*}
		Let $I_Z$ be an ideal sheaf such that the map $\sigma^\ast(I_Z)\to\sO_{\widetilde{X}}$ factors through $I_{\widetilde{Z}}$.
		Taking the composition of the maps above we get a morphism
		$$\sigma^\ast\left(\xi^\ast(\lambda_{a,r})^{\otimes k}\otimes I_Z\right)^{\otimes r}\to \sigma^\ast g^{(klr)}_\ast(rK_{Y^{(klr)}/S}).$$
		Taking the adjoint we have a morphism
		$$\alpha:\left(\xi^\ast(\lambda_{a,r})^{\otimes k}\otimes I_Z\right)^{\otimes r}\to \sigma_\ast\sigma^\ast g^{(klr)}_\ast(rK_{Y^{(klr)}/S})\stackrel{\rm trace}{\to}g^{(klr)}_\ast(rK_{Y^{(klr)}/S}).$$
		By the constructions, $\alpha|_{S^o}$ is the composition of the maps
		$$\left(\det f^o_\ast(rK_{X^o/S^o}+rB^o+raA^o)\right)^{\otimes kr}\stackrel{(\ref{align_put_polarization_to_pushforward2})}{\subset} g^o_\ast(rK_{Y^o/S^o})^{\otimes klr}\simeq g^{o(klr)}_\ast(rK_{Y^{o(klr)}/S^o}).$$
		Hence $\alpha|_{S^o}$ is injective. Since $g^{(klr)}_\ast(rK_{Y^{(klr)}/S})$ is torsion free, $\alpha$ is an injective map. Notice that $Y^{(klr)}\to S$ is smooth over $S^o$. Applying Theorem \ref{thm_abs_Arakelov} to the morphism $Y^{(klr)}\to S$ (which is smooth over $S^o$) and the torsion free sheaf $\xi^\ast(\lambda_{a,r})^{\otimes k}\otimes I_Z$ we obtain the theorem.
	\end{proof}
	\subsection{Deformation boundedness of family of lc stable minimal models}	
	Throughout this part we fix a data $$d\in\bN,c\in\bQ^{\geq0},\Gamma\subset\bQ^{\geq0}\textrm{ a finite set and }\sigma\in\bQ[t].$$
	\begin{thm}\label{thm_bounded_stable_family}
		Let $S$ be an algebraic variety such that $S_{\rm sing}$ is compact. 
		Then there is a scheme of finite type ${\bf M}$ and a morphism $S\times {\bf M}\to M_{\rm lc,(0,1)}(d,\Phi_c,\Gamma,\sigma)$ that contains all maps $S\to M_{\rm lc,(0,1)}(d,\Phi_c,\Gamma,\sigma)$ which is induced from an admissible family of $(d,\Phi_c,\Gamma,\sigma)$-stable minimal models over $S$.
	\end{thm}
	\begin{proof}
		Let $C$ be a smooth projective curve of genus $g$ and $D\subset C$ a reduced divisor. Let $\xi^o:C^o:=C\backslash D\to M_{\rm lc,(0,1)}(d,\Phi_c,\Gamma,\sigma)$ be a morphism which is induced from an admissible family of $(d,\Phi_c,\Gamma,\sigma)$-stable minimal models over $C^o$ and let $\xi:C\to M_{\rm slc}(d,\Phi_c,\Gamma,\sigma)$ be its compactification. Let $(a,r,j)\in\bQ^{\geq0}\times(\bZ^{>0})^{2}$ be a $(d,\Phi_c,\Gamma,\sigma)$-polarization data with $0\leq a\ll1$. Then one has
		\begin{align}\label{align_bounded_degree}
			\deg(\xi^\ast\lambda_{a,r})\leq \frac{rd{\rm rank}(\Lambda_{a,r})}{2}(2g+\deg(D)).
		\end{align}
		by Theorem \ref{thm_numbound_polarization}.
		Notice that $K_C+D$ may not be pseudo-effective (e.g. when $g=0$ and $\deg(D)=0,1$). We add two extra points $p_1,p_2\in C\backslash D$ so that $K_C+D'$ is always pseudo-effective for $D'=D\cup\{p_1,p_2\}$. (\ref{align_bounded_degree}) is deduced by applying Theorem \ref{thm_numbound_polarization} to the case when $S=C$ and $S^o=C\backslash (D\cup D')$.
		According to \cite[Corollary 2.23]{Kovacs2009}, the theorem follows from (\ref{align_bounded_degree}).
	\end{proof}
    \subsubsection{Strong boundedness}
	\begin{defn}
		A $(d,\Phi_c,\Gamma,\sigma)$-klt stable minimal model $(X,B),A$ is called {\bf log smooth} if $X$ is smooth and $A+B$ is a simple normal crossing $\bQ$-divisor.
		A family $f:(X,B),A\to S$ of  $(d,\Phi_c,\Gamma,\sigma)$-klt stable minimal models is called {\bf log smooth} if $X\to S$ is smooth and $A+B$ is an $f$-relative simple normal crossing divisor.
	\end{defn}
	The groupoids of log smooth families of $(d,\Phi_c,\Gamma,\sigma)$-klt stable minimal models form an open substack (denoted by $\sM_{\rm sm}(d,\Phi_c,\Gamma,\sigma)$) of $\sM_{\rm slc}(d,\Phi_c,\Gamma,\sigma)$. Its closure is a meaningful compactification of $\sM_{\rm sm}(d,\Phi_c,\Gamma,\sigma)$.
	\begin{thm}\label{thm_bounded_stable_family_strong}
		Let $S$ be an algebraic variety such that $S_{\rm sing}$ is compact. 
		Then there is a scheme of finite type  ${\bf M}$ and an admissible log smooth family $F\in\sM_{\rm sm}(d,\Phi_c,\Gamma,\sigma)(S\times {\bf M})$ of stable minimal models which contains all admissible log smooth families of  $(d,\Phi_c,\Gamma,\sigma)$-stable minimal models over $S$.
	\end{thm}
	\begin{proof}
		Let $C$ be a smooth projective curve of genus $g$ and $D\subset C$ a reduced divisor. Let $\xi^o:C^o:=C\backslash D\to M_{\rm lc,(0,1)}(d,\Phi_c,\Gamma,\sigma)$ be a morphism which is induced from a log smooth family of $(d,\Phi_c,\Gamma,\sigma)$-stable minimal models over $C^o$ and $\xi:C\to M_{\rm slc}(d,\Phi_c,\Gamma,\sigma)$ its compactification. Let $(a,r,j)\in\bQ^{\geq0}\times(\bZ^{>0})^{2}$ be a $(d,\Phi_c,\Gamma,\sigma)$-polarization data such that $0\leq a\ll 1$. Then one has
		\begin{align}\label{align_bounded_degree2}
			\deg(\xi^\ast\lambda_{a,r})\leq \frac{rd{\rm rank}(\Lambda_{a,r})}{2}(2g+\deg(D))
		\end{align}
		by Theorem \ref{thm_numbound_polarization}. Here we add two extra points to $D$ so that $K_C+D$ is always pseudo-effective.
		According to \cite[Theorem 1.7]{Kovacs2011}, the theorem follows from (\ref{align_bounded_degree2}).
	\end{proof}
	
	\bibliographystyle{plain}
	\bibliography{CGM_VHC}
	
\end{document}